\documentclass[a4paper,10pt, final]{amsart}
\usepackage[utf8x]{inputenc}
\usepackage{amsmath, amsthm, amsfonts,amssymb,epsfig,enumerate, wrapfig, tikz,  color, tensor, accents, fancybox, float, doi}
\usepackage{hyperref}
\definecolor{darkred}{rgb}{0.7,0,0}
\hypersetup{
    pdftoolbar=true,        
    pdfmenubar=true,        
    pdffitwindow=false,     
    pdfstartview={FitH},    
    pdfauthor={Catherine Gille and Louis-Hadrien Robert},     
    pdfsubject={A signature invariant for knotted Klein graphs},   
    pdfcreator={Catherine Gille and Louis-Hadrien Robert},   
    pdfproducer={Catherine Gille and Louis-Hadrien Robert}, 
    pdfkeywords={}, 
    pdfnewwindow=true,      
    colorlinks=true,       
    linkcolor=darkred,          
    citecolor=teal,        
    filecolor=magenta,      
    urlcolor=violet,          
    linkbordercolor=darkred,
    citebordercolor=teal,
    urlbordercolor=violet,  
    linktocpage=true
}
\usepackage{setspace}
\usetikzlibrary{knots}
\usetikzlibrary{arrows, calc}
\usetikzlibrary{decorations.markings}
\newcounter {res}[section]
\numberwithin{res}{section}
\newtheorem{thm}[res]{Theorem}
\newtheorem{lem}[res]{Lemma}
\newtheorem{prop}[res]{Proposition}

\theoremstyle{definition}
\newtheorem{notation}[res]{Notation}
\newtheorem{dfn}[res]{Definition}
\newtheorem{rmk}[res]{Remark}

\newcommand{\ZZ}{\ensuremath{\mathbb{Z}}} 
\newcommand{\CC}{\ensuremath{\mathbb{C}}} 

\newcommand{\RR}{\ensuremath{\mathbb{R}}} 

\newcommand{\BB}{\ensuremath{\mathbb{B}}}
\renewcommand{\SS}{\ensuremath{\mathbb{S}}} 

\newcommand{\lk}{\mathrm{lk}}
\newcommand{\id}{\mathrm{id}}

\let\oldmarginpar\marginpar
\renewcommand\marginpar[1]{\oldmarginpar{\color{red}\fbox{\begin{minipage}{3cm} \footnotesize #1 \end{minipage}}}}

\title{A signature invariant for knotted Klein graphs}
\author{Catherine Gille}
\address{Univ Paris Diderot, Institut de Mathématiques de Jussieu-Paris Rive Gauche, CNRS, Sorbonne Université. Campus des Grands Moulins, bâtiment Sophie-Germain, case 7012, 75205 Paris cedex 13, France}
\email{\href{mailto:catherine.gille@imj-prg.fr}{catherine.gille@imj-prg.fr}}
\author{Louis-Hadrien Robert}
\address{Université de Genève, 2-4 rue du Lièvre, Case postale 64, 1211 Genève 4, Switzerland}
\email{\href{mailto:louis-hadrien.robert@unige.ch}{louis-hadrien.robert@unige.ch}}
\newcommand{\kg}{\ensuremath{D_4}}
\renewcommand{\aa}{\ensuremath{a}}
\newcommand{\bb}{\ensuremath{b}}
\newcommand{\cc}{\ensuremath{c}}
\newcommand{\imagesfolder}{.}
\newcommand{\NB}[1]{\ensuremath{\vcenter{\hbox{#1}}}}
\begin{document}

\begin{abstract}
   We define some signature invariants for a class of knotted trivalent graphs using branched covers. We relate
  them to classical signatures of knots and links. Finally, we
  explain how to compute these invariants through the example of Kinoshita's knotted theta graph. 
\end{abstract}
\maketitle

\section{Introduction}
\label{sec:intro}

The notion of knotted graph generalizes the notion of link. It has
direct applications in stereo-chemistry \cite{MR1781912, MR891813}.
On the one hand, the classification of knotted graphs can be seen as an
extension of the classification problem for knots. On the other hand,
given a knotted graph, one can look at all its sub-links. Kinoshita
\cite{MR0102819, MR0312485} gave an example (see
Figure~\ref{fig:kinosym}) of a non-trivial knotted theta graph such that all the
three sub-knots are trivial (see \cite{MR3541985} for more example of
\emph{Brunnian} theta graphs). Hence it is necessary to develop specific
invariants for knotted graphs.

We restrict ourselves to a certain class of trivalent graphs in $\SS^3$ with
an edge-coloring called 3-Hamiltonian Klein graphs. The aim of this
paper is to define some signature-like invariants for knotted such graphs.

In \cite{GL78}, Gordon and Litherland explain how to compute
signatures of knots from a non-orientable spanning surface $F$ in
$\BB^4$. It is the signature of the double branched cover of
$\mathbb{B}^4$ along $F$ corrected by the normal Euler number of $F$
(see also \cite{KT76}). In this paper, we adopt this 4-dimensional point of view.

\begin{figure}[ht]
  \centering
\vspace{-0.5cm}
  \tikz{\input{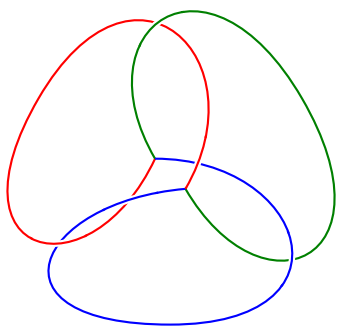}}
  \caption{Kinoshita's knotted theta graph.}
  \label{fig:kinosym}
\end{figure}

A Klein graph $\Gamma$ is a trivalent graph endowed with a 3-coloring of its edges. For any knotted Klein graph in $\SS^3$, one can
consider its \emph{Klein cover}, it is a branched $\ZZ_2\times
\ZZ_2$-cover with $\Gamma$ as branched locus. Given a spanning
foam\footnote{Foams are natural cobordisms when working with graphs:
they are surfaces with some singularities (see
Section~\ref{sec:graphs-foams}).  Here we do not suppose any kind of
orientability conditions on $F$.} $F$ for $\Gamma$, one can construct $W_F$ the
Klein cover of $\mathbb{B}^4$ along $F$. 

We define invariants of $\Gamma$ using the signature of $W_F$. For doing so, we define normal Euler numbers of foams. It turns out that if $\Gamma$ is 3-Hamiltonian\footnote{If the graph $\Gamma$ is
\emph{not} $3$-Hamiltonian, we can still define one signature
invariant, which turns out to depends only on the sub-links of
$\Gamma$.}, there are two ways to define the normal Euler numbers yielding different invariants (see Theorem~\ref{thm:main}). The computation for the Kinoshita knotted graph shows that this graph is non-trivial (and even chiral). 

Moreover, we investigate the relations between our invariants and the signatures of the different knots and links related to the knotted graph $\Gamma$. The identities we find (see Proposition~\ref{prop:weaksignaturebicolor}) can be thought of as consistency constraints between these signatures. The concept of foam enables to interpret these constraints geometrically.

For defining our invariants, we need the notion of normal Euler
numbers of immersed surfaces with boundaries. 
For this, we use linking numbers of rationally null-homologous curves in arbitrary 3-manifolds (see
\cite{MR3586621} for a gentle introduction on this notion). 
The invariance of the signatures follows from the
G-signature theorem \cite{MR0236952} in dimension 4 (see \cite{GSign}
for an elementary approach).
Finally,
in order to compute our invariants on an example, we use a result of Przytycki and
Yasuhara \cite{PY} which calculates the modification by surgery of the linking matrix of a link
in a rational homology spheres.

\subsection{Structure of the paper}
\label{sec:structure-paper}

In Section~\ref{sec:kleincovers}, we introduce the notion of Klein
graph, Klein foam and Klein cover. In Section~\ref{sec:invariants}
we define the invariants. For this we recall the notion of normal
Euler numbers in Section~\ref{sec:normal-euler-numbers-1}. The rest of
Section~\ref{sec:signature-invariant} is dedicated to the proof of
invariance: Section~\ref{sec:normal-euler-number} contains two
technical lemmas about normal Euler
numbers. Section~\ref{sec:proof-main-theorem} contains the proof of
the Theorem~\ref{thm:main}.  
Finally in Section~\ref{sec:an-example}, we compute our signature invariants
on Kinoshita's knotted graph.

\subsection{Acknowledgment}
\label{sec:acknowledgment}
The authors thank Christian Blanchet, Christine Lescop and Pierre
Vogel for enlightening conversations and Brendan Owens for his historical remarks on a previous version of this paper. L.-H.~R thanks the Université
Paris Diderot for its hospitality. L.-H.~R. was supported by the NCCR SwissMAP, funded by the Swiss National Science Foundation.

\section{Klein covers}
\label{sec:kleincovers}

\subsection{Graphs and foams}
\label{sec:graphs-foams}
Throughout the paper $\kg$ denotes the group $\ZZ_2\times \ZZ_2$ (we use the multiplication convention for the group law) and $\kg^*$ denote the set $\kg\setminus\{1\}$. The elements of $\kg^*$ are denoted by $a$, $b$ and $c$ and are represented in pictures by red, blue and green respectively. 

\begin{dfn}
  \label{dfn:graphs}
  An abstract \emph{Klein graph}  is a finite trivalent multi-graph $\Gamma$ with an edge-coloring by $\kg^*$ (as usual in graph theory we require that the colors of two adjacent edges are different). It is \emph{3-Hamiltonian} if for any two elements $i$ and $j$ of $\kg^*$, the sub-graph $\Gamma_{ij}$ consisting of edges colored by $i$ or $j$ is connected. From Section~\ref{sec:signature-invariant} on, all Klein graphs are supposed to be 3-Hamiltonian.

A \emph{knotted Klein graph} is a Klein graph $\Gamma$ together with a smooth embedding of $\Gamma$ in a manifold of dimension 3. If the manifold is not given it is assumed to be $\SS^3$.
\end{dfn}
 \begin{rmk}
   \label{rmk:graphs}
  \begin{enumerate}
  \item We should explain what is meant by \emph{smooth embedding} of an abstract Klein graph. We require that each edge is smoothly embedded and that for every vertex, none of the three tangent vector is positively co-linear with any of the two others.
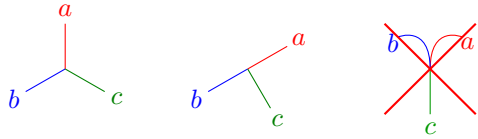
\[
\begin{tikzpicture}[scale =0.6]
  \begin{scope}
  \draw[red] (0,0) -- (90:1cm) node[pos=1.3] {$a$};
  \draw[blue] (0,0) -- (210:1cm) node[pos=1.3] {$b$};
  \draw[black!50!green] (0,0) -- (-30:1cm) node[pos=1.3] {$c$};
\end{scope}

\begin{scope}[xshift = 4cm]
  \draw[red] (0,0) -- (30:1cm) node[pos=1.3]  {$a$};
  \draw[blue] (0,0) --(210 :1cm) node[pos=1.3] {$b$};
  \draw[black!50!green] (0,0) -- (-60:1cm) node[pos=1.3] {$c$};
\end{scope}

\begin{scope}[xshift = 8cm]
  \draw[red] (0,0)  .. controls +(0,1) and +(0,0) .. (45:1cm); 
  \node[red] at (35:1cm) {$a$};
  \draw[blue] (0,0) .. controls +(0,1) and +(0,0) ..  (135 :1cm); 
  \node[blue] at (145:1cm) {$b$};
  \draw[black!40!green] (0,0) -- (-90:1cm) node[pos=1.3] {$c$};
\draw[thick, red] (-1,-1) -- (1,1);
\draw[thick, red] (-1,1) -- (1,-1);
\end{scope}
\end{tikzpicture}
\]
\item We consider embedded Klein graphs up to ambient isotopy. In a diagrammatic approach, this means that a graph is considered up to the classical Reidemeister moves and the following additional ones (see \cite[Proposition 1.6]{lewar2013}):
\begin{align*}
\begin{array}{cc}
\NB{\begin{tikzpicture}[scale =0.7]
  \begin{scope}
  \draw (0,0.5) -- (90:1cm);
  \draw (0,0.5) -- (210:1cm);
  \draw (0,0.5) -- (-30:1cm);
\end{scope}

\node at (2, 0) {$\leftrightsquigarrow$};
\node at (-2, 0) {$\leftrightsquigarrow$};

\begin{scope}[xshift = 4cm]
  \draw (0,0.5) -- (90:1cm) ;
  \draw (0,0.5) .. controls (-0:0.7) and  +(30:0.2cm) .. (210 :1cm);
  \fill[white] (0, -0.06) circle (0.07);
  \draw (0,0.5)  .. controls (180:0.7cm) and +(150: 0.2cm) .. (-30:1cm);
\end{scope}

\begin{scope}[xshift = -4cm]
  \draw (0,.5) -- (90:1cm) ;
  \draw (0,.5)  .. controls (180:0.7cm) and +(150: 0.2cm) .. (-30:1cm);
  \fill[white] (0, -0.06) circle (0.07);
  \draw (0,.5) .. controls (0:0.7) and  +(30:0.2cm) .. (210 :1cm);
\end{scope} 
\end{tikzpicture}}& \textrm{(Rv1)} \\
\NB{\begin{tikzpicture}[scale =0.7]
  \begin{scope}
  \draw (0,0) -- (90:1cm);
  \draw (0,0) -- (210:1cm);
  \draw (0,0) -- (-30:1cm);
  \fill[white] (90:0.5cm) circle (0.07);
  \draw (180:1cm)  .. controls +(0.3,0) and +(-0.3, 0) .. (90: 0.5) .. controls +(0.3,0) and +(-0.3, 0) .. (0:1cm);
\end{scope}
\node at (2, 0) {$\leftrightsquigarrow$};
\begin{scope}[xshift = 4cm]
  \draw (0,0) -- (90:1cm);
  \draw (0,0) -- (210:1cm);
  \draw (0,0) -- (-30:1cm);
  \fill[white] (-30:0.54) circle (0.07);
  \fill[white] (210:0.54) circle (0.07);
  \draw (180:1cm)  .. controls +(0.3,0) and +(-0.3, 0) .. (-90: 0.5) .. controls +(0.3,0) and +(-0.3, 0) .. (0:1cm);
\end{scope}

\begin{scope}[xshift=9cm]
\begin{scope}
  \draw (180:1cm)  .. controls +(0.3,0) and +(-0.3, 0) .. (90: 0.5) .. controls +(0.3,0) and +(-0.3, 0) .. (0:1cm);
  \fill[white] (90:0.5cm) circle (0.07);
  \draw (0,0) -- (90:1cm);
  \draw (0,0) -- (210:1cm);
  \draw (0,0) -- (-30:1cm);
\end{scope}
\node at (2, 0) {$\leftrightsquigarrow$};
\begin{scope}[xshift = 4cm]
  \draw (180:1cm)  .. controls +(0.3,0) and +(-0.3, 0) .. (-90: 0.5) .. controls +(0.3,0) and +(-0.3, 0) .. (0:1cm);
  \fill[white] (-30:0.54) circle (0.07);
  \fill[white] (210:0.54) circle (0.07);
  \draw (0,0) -- (90:1cm);
  \draw (0,0) -- (210:1cm);
  \draw (0,0) -- (-30:1cm);
\end{scope}
\end{scope} 
\end{tikzpicture}}&\textrm{(Rv2)} 
\end{array}
\end{align*}
\item One may wonder which trivalent graphs can be endowed with a structure of Klein graphs. It is known to be the case for planar graphs with no bridge (this follows from the $4$-color theorem) and for bipartite graphs (this follows from König's theorem). However, the 3-Hamiltonian condition is more complicated to ensure. Four examples are given in Figure~\ref{fig:examplhamilton}.
\begin{figure}[ht]
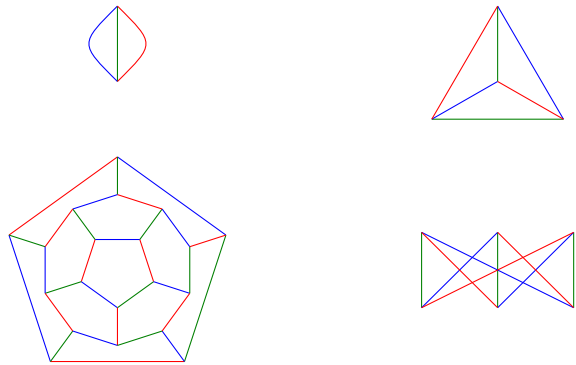

 \centering
  \begin{tikzpicture}
    \begin{scope}[xshift = 0cm, yshift = 2.5cm]
      \input{\imagesfolder/st_theta}
    \end{scope}
    \begin{scope}[xshift = 5cm, yshift= 2.5cm]
      \input{\imagesfolder/st_tetraedra}
    \end{scope}
    \begin{scope}[scale = 0.5, xshift = 0cm]
      \begin{scope}[yscale = {1}, xscale={1},decoration={markings, mark=at
  position 0.5 with {\arrow{>}}},postaction={decorate}]
\draw[color=blue] (18:3) -- (90:3);
\draw[color=red] (90:3) -- (162:3);
\draw[color=blue](162:3) -- ( 234:3);
\draw[color=red]( 234:3) -- (306:3);
\draw[color = green!50!black](306:3) -- (18:3);
\draw[color=red] (18:3) -- (18:2);
\draw[color=green!50!black] (90:3) -- (90:2);
\draw[color=green!50!black](162:3) -- (162:2);
\draw[color=green!50!black]( 234:3) -- (234:2);
\draw[color = blue](306:3) -- (306:2);
\draw[color=red] (-18:1) -- (54:1);
\draw[color=blue] (54:1) -- (126:1);
\draw[color=red](126:1) -- ( 198:1);
\draw[color=blue]( 198:1) -- (270:1);
\draw[color = green!50!black](270:1) -- (-18:1);
\draw[color=blue] (-18:1) -- (-18:2);
\draw[color=green!50!black] (54:1) -- (54:2);
\draw[color=green!50!black](126:1) -- (126:2);
\draw[color=green!50!black]( 198:1) -- ( 198:2);
\draw[color = red](270:1) -- (270:2);
\draw[color=green!50!black] (-18:2) -- (18:2);
\draw[color=blue] (18:2) -- (54:2);
\draw[color=red](54:2) -- (90:2);
\draw[color=blue]( 90:2) -- (126:2);
\draw[color = red](126:2) -- (162:2);
\draw[color=blue] (162:2) -- (198:2);
\draw[color=red] (198:2) -- (234:2);
\draw[color=blue](234:2) -- (270:2);
\draw[color=green!50!black]( 270:2) -- (306:2);
\draw[color = red](306:2) -- (-18:2);
\end{scope}
    \end{scope}
    \begin{scope}[xshift = 5cm, yshift = -0.5cm]
      \input{\imagesfolder/st_utilitygraph}
    \end{scope}
  \end{tikzpicture}
  \caption{Examples of 3-Hamiltonian Klein graphs.}
  \label{fig:examplhamilton}
\end{figure}
\item \label{item:connsum}The 3-Hamiltonian condition is preserved by performing connected sum along a vertex. This operation is described in Figure~\ref{fig:connsum}.
\end{enumerate}
\end{rmk}

  \begin{figure}[ht]
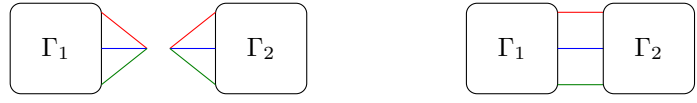

    \centering
    \NB{\tikz[scale=0.6]{\begin{scope}
  \begin{scope}
    \begin{scope}
      \draw[rounded corners] (-1,-1) -- (-1,1) -- (1,1) -- (1, -1) --cycle; 
      \draw[red] (1,0.8) -- (2,0); 
      \draw[blue] (1,0) -- (2,0); 
      \draw[green!50!black] (1,-0.8) -- (2,0); 
      \node at (0,0) {$\Gamma_1$};
    \end{scope}
    \begin{scope}[xshift = 4.5cm, xscale =-1]
      \draw[rounded corners] (-1,-1) -- (-1,1) -- (1,1) -- (1, -1) --cycle; 
      \draw[red]  (1,0.8) -- (2,0); 
      \draw[blue] (1,0)   --  (2,0); 
      \draw[green!50!black] (1,-0.8) -- (2,0); 
      \node at (0,0) {$\Gamma_2$};
    \end{scope}
  \end{scope}

  \begin{scope}[xshift=10cm]
    \begin{scope}
      \draw[rounded corners] (-1,-1) -- (-1,1) -- (1,1) -- (1, -1) --cycle; 
      \draw[red] (1,0.8) -- +(1,0); 
      \draw[blue] (1,0) -- +(1,0); 
      \draw[green!50!black] (1,-0.8) -- +(1,0); 
      \node at (0,0) {$\Gamma_1$};
    \end{scope}
    \begin{scope}[xshift = 3cm]
      \draw[rounded corners] (-1,-1) -- (-1,1) -- (1,1) -- (1, -1) --cycle;

      \node at (0,0) {$\Gamma_2$};
    \end{scope}
  \end{scope}

\end{scope}}} 
    \caption{Connected sum along a vertex.}
    \label{fig:connsum}
  \end{figure}

\begin{dfn}
  \label{dfn:Kleinfoam}
  A \emph{closed embedded Klein foam} $E$ is the realization of a finite CW-complex in a manifold of dimension 4 (if the manifold is not given, it is meant to be $\SS^4$) and some data attached to it. For every point of the CW complex, we require that there exists a neighborhood which is either diffeomorphic to a disk or to following picture (the singular one-dimensional cell is called a \emph{binding}):
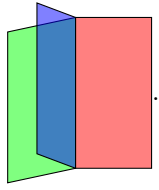
\[
\NB{\begin{tikzpicture}
\begin{scope}

  \filldraw[fill=red,fill opacity = 0.5, draw= black, draw opacity=1] (0,0,0) rectangle (1,2,0);
\filldraw[fill=green,fill opacity = 0.5, draw= black, draw opacity=1] (0,0,0) -- (-0.7,0,0.5) -- (-0.7,2,0.5) -- (0,2,0);
\filldraw[fill=blue,fill opacity = 0.5, draw= black, draw opacity=1] (0,0,0) -- (0,2,0) -- (-0.7,2,-0.50) -- (-0.7,0,-0.50) -- cycle ;
\end{scope}
\end{tikzpicture}}.
\]  
The data attached to $E$ is a coloring of its facets by $\kg^*$, such that the three facets adjacent to a common binding have different colors. 

  The intersection $F$ of an embedded Klein foam $E$ with a sub-manifold of dimension 4 with boundary
is \emph{a Klein foam with boundary} if:
  \begin{itemize}
  \item the intersection of $E$ with $\partial W$ is a knotted Klein graph $\Gamma$ in $\partial W$,
  \item there exists a tubular neighborhood $U$ of $\partial W$ such that $(U, E \cap U)$ is diffeomorphic to $(\partial W \times ]0,1[, \Gamma \times ]0,1[)$.
  \end{itemize}
In this case, $\Gamma$ is the \emph{boundary of $F$} and we write $\Gamma= \partial F$. We say as well that $F$ is a \emph{spanning foam} for $\Gamma$.
\end{dfn}

\begin{prop}\label{prop:spanningfoamexists}[Proof in Appendix~\ref{sec:spanning-foams}]
  Let $\Gamma$ be a knotted Klein graph in $\SS^3$. There exists a spanning foam for $\Gamma$ in $\BB^4$.
\end{prop}

\subsection{Klein covers}
\label{sec:klein-cover}
\begin{dfn}
  \label{dfn:kleinaction}
  Let $M$ be a closed oriented manifold. If $\kg$ acts on $M$ by
  positive diffeomorphisms and if for every $g$ in $\kg^*$, the set $M^g$ of
  fix points of $g$ is a sub-manifold of co-dimension 2, we say that
  $M$ is a \emph{Klein manifold}. 
 We set
  $M^{\cup\kg}:= \cup_{g\in \kg^*} M^g$.

\end{dfn}

\begin{prop}
  \label{prop:quotientismnf}
  Let $M$ be a Klein manifold of dimension 3 (resp. of dimension 4), then $M/\kg$ is a closed oriented manifold of the same dimension and $M^{\cup \kg}$ is mapped on a Klein graph (resp. on a Klein foam) by $\pi: M \twoheadrightarrow M/\kg$.
\end{prop}

\begin{proof}
  Dimension 3 and 4 are analogous. We only treat dimension 3.
  Let $g_1$ be an element of $\kg^*$ and $x$ in $M$ be a fixed point for $g_1$. The diffeomorphism induced by $g_1$ being an involution, the action of $g_1$ on $T_xM$ is diagonalizable and we can find a basis of $T_xM$ such that the matrices of the linear map induced by $g_1$ is equal to
\[
  \begin{pmatrix}
    1 & 0 &0 \\ 0 & -1 & 0 \\ 0 & 0 & -1
  \end{pmatrix}.
\]

Let us call $g_2$ and $g_3$ the two remaining elements of $\kg$. We have $g_2\cdot x  = g_2g_1 \cdot x = g_3 \cdot x$. Hence, if $x$ is fixed by $g_2$ if and only if it is fixed by $g_3$.

Suppose that $x$ is not fixed by $g_2$. Then the restriction of $\pi:M \to  M/\kg$ to a neighborhood of $\{ x, g_2\cdot x\}$ is isomorphic to a double branched cover followed by a trivial $2$-folds cover. 

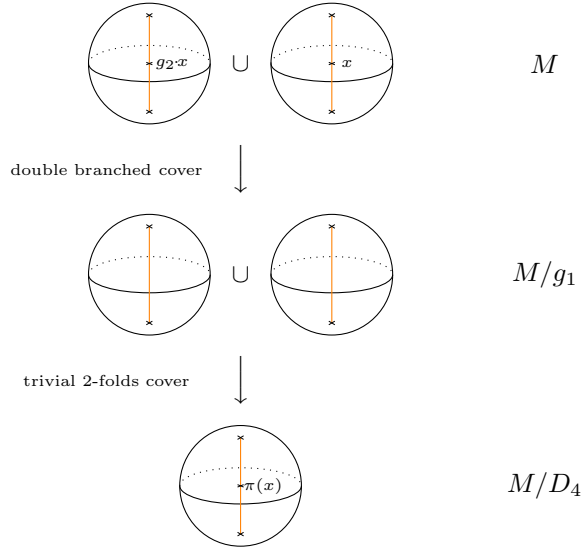
\[
  \begin{tikzpicture}[scale =0.8]
    \begin{scope}[rotate = -90]
\begin{scope}[xshift = - 3.5cm, yshift=-1.5cm, rotate=90]
\draw (0,0) circle (1cm);
\draw (-1,0) arc (-180:0: 1cm and 0.3cm);
\draw[dotted] (-1,0) arc (180:0: 1cm and 0.3cm);
\draw[orange] (0, 0.8) -- (0, -0.8);
\draw [very thin] (0.05, 0.77) -- (-0.05, 0.83);
\draw [very thin] (-0.05, 0.77) -- (0.05, 0.83);
\draw [very thin] (0.05, -0.77) -- (-0.05, -0.83);
\draw [very thin] (-0.05, -0.77) -- (0.05, -0.83);
\draw [very thin] (0.05, -0.02) -- (-0.05, 0.02);
\draw [very thin] (-0.05, -0.02) -- (0.05, 0.02);
\node at (0.37, 0) {\tiny{$g_2\!\!\cdot\!\! x$}};
\end{scope}
\node at (-3.5,0) {$\cup$}; 
\node at (-3.5, 5) {$M$}; 
\begin{scope} [xshift = - 3.5cm, yshift= +1.5cm, rotate=90]
\draw (0,0) circle (1cm);
\draw (-1,0) arc (-180:0: 1cm and 0.3cm);
\draw[dotted] (-1,0) arc (180:0: 1cm and 0.3cm);
\draw[orange] (0, 0.8) -- (0, -0.8);
\draw [very thin] (0.05, 0.77) -- (-0.05, 0.83);
\draw [very thin] (-0.05, 0.77) -- (0.05, 0.83);
\draw [very thin] (0.05, -0.77) -- (-0.05, -0.83);
\draw [very thin] (-0.05, -0.77) -- (0.05, -0.83);
\draw [very thin] (0.05, -0.02) -- (-0.05, 0.02);
\draw [very thin] (-0.05, -0.02) -- (0.05, 0.02);
\node at (0.25, 0) {\tiny $x$};
\end{scope}
\node at (-1.75,0) {$\Big\downarrow$};
\node at (-1.75,-2.2) {\tiny{double branched cover}};
\begin{scope}[xshift = - 0cm, yshift=-1.5cm, rotate=90]
\draw (0,0) circle (1cm);
\draw (-1,0) arc (-180:0: 1cm and 0.3cm);
\draw[dotted] (-1,0) arc (180:0: 1cm and 0.3cm);
\draw[orange] (0, 0.8) -- (0, -0.8);
\draw [very thin] (0.05, 0.77) -- (-0.05, 0.83);
\draw [very thin] (-0.05, 0.77) -- (0.05, 0.83);
\draw [very thin] (0.05, -0.77) -- (-0.05, -0.83);
\draw [very thin] (-0.05, -0.77) -- (0.05, -0.83);
\end{scope}
\node at (0,0) {$\cup$};
\node at (0, 5) {$M/{g_1}$}; 
\begin{scope} [xshift =  0cm, yshift=1.5cm, rotate=90]
\draw (0,0) circle (1cm);
\draw (-1,0) arc (-180:0: 1cm and 0.3cm);
\draw[dotted] (-1,0) arc (180:0: 1cm and 0.3cm);
\draw[orange] (0, 0.8) -- (0, -0.8);
\draw [very thin] (0.05, 0.77) -- (-0.05, 0.83);
\draw [very thin] (-0.05, 0.77) -- (0.05, 0.83);
\draw [very thin] (0.05, -0.77) -- (-0.05, -0.83);
\draw [very thin] (-0.05, -0.77) -- (0.05, -0.83);
\end{scope}
\node at (1.75,0) {$\Big\downarrow$};
\node at (1.75,-2.2) {\tiny{trivial 2-folds cover}};
\begin{scope} [xshift = 3.5cm, yshift= 0cm, rotate=90]
\draw (0,0) circle (1cm);
\draw (-1,0) arc (-180:0: 1cm and 0.3cm);
\draw[dotted] (-1,0) arc (180:0: 1cm and 0.3cm);
\draw[orange] (0, 0.8) -- (0, -0.8);
\draw [very thin] (0.05, 0.77) -- (-0.05, 0.83);
\draw [very thin] (-0.05, 0.77) -- (0.05, 0.83);
\draw [very thin] (0.05, -0.77) -- (-0.05, -0.83);
\draw [very thin] (-0.05, -0.77) -- (0.05, -0.83);
\draw [very thin] (0.05, -0.02) -- (-0.05, 0.02);
\draw [very thin] (-0.05, -0.02) -- (0.05, 0.02);
\node at (0.37, 0) {\tiny $\pi(x)$};
\end{scope}
\node at (3.5, 5) {$M/\kg$}; 
\end{scope}
  \end{tikzpicture}
\]
This implies that $\pi(x)$ has a neighborhood homeomorphic to a ball. 

Suppose now that  $x$ is fixed by all the elements of $\kg^*$. We look at the action of $\kg$  over $T_xM$. This can be seen as a map from $\phi: \kg \to GL_3(\RR)$. The matrices $\phi(g_1), \phi(g_2)$ and $\phi(g_3)$ are simultaneously diagonalizable and we can find a basis of $T_xM$ such that these matrices are equal to:
\[
 \begin{pmatrix}
    1 & 0 &0 \\ 0 & -1 & 0 \\ 0 & 0 & -1
  \end{pmatrix}, \quad 
 \begin{pmatrix}
    -1 & 0 &0 \\ 0 & 1 & 0 \\ 0 & 0 & -1
  \end{pmatrix}\quad \textrm{and } \quad
 \begin{pmatrix}
    -1 & 0 &0 \\ 0 & -1 & 0 \\ 0 & 0 & 1
  \end{pmatrix}
\]
Hence, in a chart, the action of $g_1$, $g_2$ and $g_3$ is given by these matrices and the fixed point loci looks like
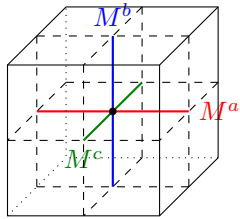
\[
  \begin{tikzpicture}[scale =1]
    \begin{scope}
  \draw[thick,red] (-1,0,0) -- (1,0,0) node[right] {$M^{a}$};
  \draw[thick, blue] (0,-1,0) -- (0,1,0) node[above] {$M^{b}$};
  \draw[thick, green!50!black] (0,0,-1) -- (0,0,1) node[below] {$M^{c}$};
  \fill (0,0,0) circle (0.5mm);
\draw(-1,-1,1) -- (-1,1,1) -- (1,1,1) -- (1,-1,1)-- cycle;
\draw[dotted]  (-1,-1, 1) -- +(0,0,-2);
\draw  (-1, 1, 1) -- +(0,0,-2);
\draw  ( 1, 1, 1) -- +(0,0,-2);
\draw  ( 1,-1, 1) -- +(0,0,-2);
\draw[dotted] ( 1,-1,-1) -- (-1,-1,-1) -- (-1, 1,-1);
\draw (-1,1,-1) -- ( 1, 1,-1) -- (1,-1,-1);
\draw[dashed] (-1, 1,0) -- +(2,0,0);
\draw[dashed] (-1,-1,0) -- +(2,0,0);
\draw[dashed] (-1,0, 1) -- +(2,0,0);
\draw[dashed] (-1,0,-1) -- +(2,0,0);

\draw[dashed] (1, 0,-1) -- +(0,0,2);
\draw[dashed] (-1,0,-1) -- +(0,0,2);
\draw[dashed] (0, 1,-1) -- +(0,0,2);
\draw[dashed] (0,-1,-1) -- +(0,0,2);

\draw[dashed] ( 1,-1, 0) -- +(0,2,0);
\draw[dashed] (-1,-1, 0) -- +(0,2,0);
\draw[dashed] ( 0,-1, 1) -- +(0,2,0);
\draw[dashed] ( 0,-1,-1) -- +(0,2,0);

\end{scope}
  \end{tikzpicture}
\]
in a neighborhood of $x$. This implies that a neighborhood of $\pi(x)$ in $M/\kg$ is given by gluing the two cubes
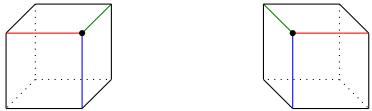
\[
  \begin{tikzpicture}[scale =1]
    \begin{scope}
  \draw[dotted] (0,0,0) -- (0,0,1);
  \draw (0,0,1) -- (0,1,1);
  \draw (0,1,1) -- (0,1,0) ;
  \draw[dotted] (0,1,0) -- (0,0,0);
  \draw (1,0,0) -- (1,0,1);
  \draw[blue] (1,0,1) -- (1,1,1);
  \draw[green!50!black] (1,1,1) -- (1,1,0); 
  \draw (1,1,0) -- (1,0,0);
  \draw[dotted] (0,0,0) -- +(1,0,0);
  \draw (0,1,0) -- +(1,0,0);
  \draw (0,0,1) -- +(1,0,0);
  \draw[red] (0,1,1) -- +(1,0,0);
  \fill (1,1,1) circle (0.04);
\end{scope}

\begin{scope} [xshift = 4cm, xscale =-1]
  \draw[dotted] (0,0,0) -- (0,0,1);
  \draw (0,0,1) -- (0,1,1);
  \draw (0,1,1) -- (0,1,0) ;
  \draw[dotted] (0,1,0) -- (0,0,0);
  \draw (1,0,0) -- (1,0,1);
  \draw[blue] (1,0,1) -- (1,1,1);
  \draw[green!50!black] (1,1,1) -- (1,1,0); 
  \draw (1,1,0) -- (1,0,0);
  \draw[dotted] (0,0,0) -- +(1,0,0);
  \draw (0,1,0) -- +(1,0,0);
  \draw (0,0,1) -- +(1,0,0);
  \draw[red] (0,1,1) -- +(1,0,0);
  \fill (1,1,1) circle (0.04);
\end{scope}
  \end{tikzpicture}
\]
along the three faces containing the distinguished vertex. This gives a 3-dimensional ball. This proves that $M/\kg$ is a manifold and that in a neighborhood of $\pi(x)$ the image of $\pi(M^{\cup \kg})$ in is given in a chart by:
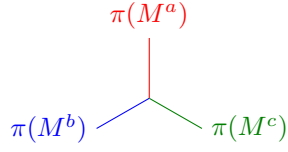
\[
  \begin{tikzpicture}[scale =0.8]
    \begin{scope}
  \draw[red] (0,0) -- (90:1cm) node[above] {$\pi(M^{a})$};
  \draw[blue] (0,0) -- (210:1cm) node[left] {$\pi(M^{b})$};
  \draw[black!50!green] (0,0) -- (-30:1cm) node[right] {$\pi(M^{c})$};
\end{scope}
  \end{tikzpicture}
\]
Since $M$ is compact, there are finitely many points which are fixed by the whole group $\kg$. Altogether, this implies that $M/\kg$ is a manifold and that the image of $M^{\cup \kg}$ is a Klein graph. For the 4-dimensional statements, all local discussions have to be multiplied by an interval.
\end{proof}

\begin{dfn}
  \label{dfn:kleincover}
  Let $M$  be a closed manifold of dimension $3$ (resp. of dimension $4$) and $\Gamma$ (resp. $F$) be an embedded Klein graph (resp. a Klein foam). Suppose that $N$ is a Klein manifold of the same dimension such that $N/\kg \simeq M$ and that $\pi(N^{\cup \kg})$ is identified with $\Gamma$ (resp. $F$). Then we say that $N$ is a \emph{Klein cover of $M$ along $\Gamma$ (resp. along $F$)}. 
\end{dfn}

\begin{prop}
  \label{prop:kleincoverexists}
  \begin{enumerate}
  \item For any embedded Klein graph $\Gamma$ in $\SS^3$, there exists a unique (up to diffeomorphism) Klein cover of $\SS^3$ along $\Gamma$. It is denoted by $M_\Gamma$.
\item For any Klein foam $F$ embedded in $\SS^4$, there exists a unique (up to diffeomorphism) Klein cover of $\SS^4$ along $F$. It is denoted by $V_F$.
  \item For any Klein foam $F$ properly embedded in $\BB^4$, there exists a unique (up to diffeomorphism) Klein cover of $\BB^4$ along $F$. It is denoted by $W_F$. (The first point actually implies that $\partial W_F \simeq M_\partial F$).
  \end{enumerate}
\end{prop}

\begin{proof}
  We only prove the first statement, the two others are analogous. Fix an arbitrary orientation on $\Gamma = (V,E)$. The first homology group of $M:=\SS^3 \setminus \Gamma$ is generated by elements $[\gamma_e]_{e\ \in E}$, where $\gamma_e$ is a small loop wrapping positively around the edge $e$. More precisely 
\[ H_1(M, \ZZ)\simeq \left.\left( \bigoplus_{z \in E}\ZZ \right) \right/  \left\langle  \pm[\gamma_{e_1(v)}] + \pm[\gamma_{e_2(v)}] + \pm[\gamma_{e_3(v)}] =0 | \textrm{for all $v$ in $V$}\right\rangle \]
where $e_1(v), e_2(v)$ and $e_3(v)$ are the three edges adjacent to $v$ and the signs ambiguity is given by the way the orientations are toward $v$. The coloring of $\Gamma$ gives a morphism
\[
\phi:H_1(M) \to \kg
\]
which sends $[\gamma_e]$ to the color of $e$ in $\Gamma$. Note here that the ambiguity in the orientation of $\Gamma$ is not a problem since all non-trivial elements of $\kg$ have order $2$. This is indeed a morphism, since in $\kg$ we have $g_1g_2g_3 =1$. 
We can promote $\phi$ to a morphism:
\[\phi : \pi_1(M) \to \kg.\]
We consider the covering $M'$ associated with this morphism. We can complete $M'$ in order to make it a Klein cover of $\SS^3$. For the edges and the vertices we use the local models described in the proof of Proposition~\ref{prop:quotientismnf}.

Uniqueness comes from the uniqueness of the covering $M'$ and from the fact that the local models for the singular points and the edges are the only possible ones (see proof of \ref{prop:quotientismnf}).
The proof in dimension 4 is analogous. Note that in this case, the local models are given by the ones we described times an interval. 

\end{proof}

\section{A signature invariant}
\label{sec:signature-invariant}

\subsection{Normal Euler numbers}
\label{sec:normal-euler-numbers-1}
The aim of this part is to recall the definition of normal Euler numbers of surfaces with boundaries (Definition~\ref{dfn:QHSeulernumbers}) and to give an additivity property that they satisfy. They have been studied by Gilmer~\cite{MR1238876} and require linking numbers of rationally null-homologous knots in arbitrary 3-manifolds. We refer to the lecture notes of Lescop \cite[Section 1.5]{MR3586621} for details about such linking numbers.
We start with the normal Euler number of a closed surface:
\begin{dfn}
  \label{dfn:NENclosedembeddedsurface}

  Let $\Sigma$ be a smooth, not necessarily orientable, closed surface immersed in a smooth oriented manifold of dimension 4 and let $s$ be a section of the normal bundle of $\Sigma$ transverse to the null section $s_0$. Then the intersection (computed with local orientations) of $s_0$ with $s$ is called the \emph{normal Euler number of $\Sigma$} and is denoted by $e(\Sigma)$. As the definition suggests the integer $e(\Sigma)$ does not depends on $s$.
\end{dfn}
\begin{rmk}
  \label{rmk:eulerclass}
  If we choose a local system of orientations $O(\Sigma)$ of $\Sigma$, we can define the Euler class  $e(N_\Sigma)$ of the normal bundle $N_\Sigma$. This is an element of $H^2(\Sigma, O(\Sigma))$. We consider $[\Sigma] \in H_2(\Sigma, O(\Sigma))$, the fundamental class of $\Sigma$ in this local system of orientations. The normal Euler number of $\Sigma$ is then equal to $e(N_\Sigma)([\Sigma])$.
\end{rmk}

If the $4$-manifold  and the surface has boundary, we can extend this definition relatively to the boundary.

\begin{dfn}[Relative normal Euler number]
  \label{dfn:relative_euler_number}
  Let $W$ be an oriented $4$-manifold with boundary and $\Sigma$ be a smooth, properly immersed surface with boundary in $W$. 
Let $L= l_1\cup \dots \cup l_k$ be the boundary of $\Sigma$. Let us choose $\tilde{l}_1, \dots, \tilde{l}_k$ some parallels of $l_1, \dots, l_k$ in $\partial W$. The normal Euler number of $\Sigma$ relatively to $\tilde{l}_1, \dots, \tilde{l}_k$ is the intersection number of $\Sigma$ with a section $s$  of the normal bundle of $\Sigma$ (transverse to the null section $s_0$), such that $\partial s= \tilde{l}_1\cup \dots\cup \tilde{l}_k$. We denote it by $e(\Sigma; \tilde{l}_1,\dots, \tilde{l}_k)$.
\end{dfn}

\begin{prop}[{\cite[p. 311]{MR1238876}}]
  \label{prop:diff-relative-euler-numbers}
  Let $W$ be an oriented $4$-manifold with boundary and $\Sigma$ be a smooth, properly immersed surface with boundary. 
Let $L= l_1\cup \dots \cup l_k$ be the boundary of $\Sigma$. Suppose that each $l_i$ is rationally null-homologous. 
Let us choose $\tilde{l}^1_1, \dots, \tilde{l}^1_k$  and $\tilde{l}^2_1, \dots, \tilde{l}^2_k$ two sets of parallels of $l_1, \dots, l_k$. 
For each $i$ pick an orientation of $l_i$ and orient $\tilde{l}^1_i$ and $\tilde{l}^2_i$ accordingly.
Then for every $i$ in $[1, l]$, $\mathrm{lk}(l_i, \tilde{l}^1_i) - \mathrm{lk}(l_i, \tilde{l}^2_i)$ is an integer, and we have:
\[
e(\Sigma; \tilde{l}^2_1,\dots, \tilde{l}^2_k) - 
e(\Sigma; \tilde{l}^1_1,\dots, \tilde{l}^1_k) = \sum_{i=1}^k (\mathrm{lk}(l_i, \tilde{l}^2_i) - \mathrm{lk}(l_i, \tilde{l}^1_i)).
\]
\end{prop}

\begin{proof}
  Let $T_i$ be a tubular neighborhood of $l_i$ in $\partial W$. An homotopy from $\tilde{l_i}^1$ to $\tilde{l_i}^2$ can be though of a section $s_{h_i}$ of the normal bundle of $l_i\times [-\epsilon, \epsilon]$ in $T_i\times [-\epsilon,
 \epsilon]$. The intersection of $s_{h_i}(l_i\times [-\epsilon, \epsilon])$ with $l_i\times [-\epsilon,\epsilon]$ is equal to $ \mathrm{lk}(l_i, \tilde{l}^1_i) - \mathrm{lk}(l_i, \tilde{l}^2_i)$.
Let $s_1$ be a section of the normal bundle of $\Sigma$ used to compute $e(\Sigma; \tilde{l}^1_1,\dots, \tilde{l}^1_k)$. Gluing the $s_{h_i}$'s to $s_1$, one obtains a section $s_2$ which bounds $\tilde{l}^1_1\cup\dots \cup \tilde{l}^1_k$. 
Hence: \[\Sigma\cap s_2(\Sigma) - \Sigma \cap s_1(\Sigma) = \sum_{i=1}^k (\mathrm{lk}(l_i, \tilde{l}^2_i) - \mathrm{lk}(l_i, \tilde{l}^1_i)). \]
\end{proof}
This yields the following definition:
\begin{dfn}[Normal Euler numbers, boundary case]
  \label{dfn:QHSeulernumbers}
  Let $W$ be an oriented $4$-manifold with boundary and $\Sigma$ be a smooth surface with boundary properly immersed. 
    Suppose that every component $l_1, \dots, l_k$ of the boundary of $\Sigma$ is rationally null-homologous, and let $s$ be a section of the normal bundle of $\Sigma$ (transverse to $s_0$) and denote $l'_1, \dots, l'_k$ the parallels of $l_1, \dots, l_k$ induced by $s$. We define $e(\Sigma) = \Sigma \cap s(\Sigma) - \sum_{i=1}^k \mathrm{lk}(l_i, l'_i)$.
\end{dfn}
\begin{rmk} \label{rmk:NENwithlk}
  \begin{enumerate}
  \item    \label{prop:eulerlinking}
Suppose that every  component of $\Sigma$ has a non-empty boundary. Then it is possible to find a nowhere vanishing section $s$ of the normal bundle of $\Sigma$. 
Let us denote $l'_1, \dots, l'_k$ the parallels of $l_1,\dots, l_k$ induced by $s$. 
Then we have: 
\[e(\Sigma) =  -\sum_{i=1}^k \mathrm{lk}(l_i, l'_i).\]
\item If $\partial W$ is a rational homology sphere, then the conditions on the $l_i$'s are automatically satisfied.
\end{enumerate}
\end{rmk}
\begin{prop}
  \label{prop:en-change-orientation-and-glue}
  Let $W$ be an oriented $4$-manifold with boundary and $\Sigma$ be a smooth surface with boundary properly immersed in $W$.
Suppose that each connected component of $\partial \Sigma$ is rationally null-homologous.  \begin{itemize}
  \item The normal Euler number $e(\Sigma, -W)$ of  $\Sigma$ in the manifold $W$ endowed with the opposite orientation is equal to $-e(\Sigma,W)$.
  \item Let $W'$ be an oriented $4$-manifold with boundary and $\Sigma'$ be a smooth surface with boundary immersed in $W'$ such that $\partial \Sigma' \subseteq \partial W'$. Suppose that $\phi$ is an orientation-reversing diffeomorphism from $\partial W$ to $\partial W'$ which maps $\partial \Sigma$ on $\partial \Sigma'$. Then $e(\Sigma \cup_\phi \Sigma', W\cup_\phi W')= e(\Sigma, W) + e(\Sigma', W')$. Note that $\Sigma \cup_\phi \Sigma'$ is a closed surface in a closed oriented manifold.
  \end{itemize}
\end{prop}

\begin{proof}
  The first assertion directly follows from the definition. For the second one, we choose a set of parallels $\tilde{l}_1,\dots, \tilde{l}_k$ for $\partial \Sigma$ and a section $s$ of the normal bundle of $\Sigma$ bounding them. We choose a section $s'$ of the normal bundle of $\Sigma'$ bounding $\phi(\tilde{l}_1), \dots, \phi(\tilde{l}_k)$. We can glue $s$ and $s'$ along $\phi$, this gives a section $s''$ of the normal bundle of $\Sigma'\cup_\phi \Sigma$. Hence we have:
\begin{align*}
e(\Sigma'\cup_\phi \Sigma) &= (\Sigma'\cup_\phi \Sigma)\cap s''(\Sigma'\cup_\phi \Sigma) \\
&=\Sigma \cap s(\Sigma) + \Sigma' \cap s'(\Sigma') \\
&=\Sigma \cap s(\Sigma) - \sum_{i=1}^k\mathrm{lk}_{\partial W}(l_i, \tilde{l}_i) + \sum_{i=1}^k\mathrm{lk}_{\partial W}(l_i, \tilde{l}_i) + \Sigma' \cap s(\Sigma') \\
&=\Sigma \cap s(\Sigma) - \sum_{i=1}^k\mathrm{lk}_{\partial W}(l_i, \tilde{l}_i) - \sum_{i=1}^k\mathrm{lk}_{\partial W'}(\phi(l_i), \phi(\tilde{l}_i)) + \Sigma' \cap s(\Sigma') \\
&=e(\Sigma, W) + e(\Sigma', W').
\end{align*}
\end{proof}
\subsection{The invariants}
\label{sec:invariants}
We start by introducing some notations which will be used throughout the rest of the paper.
\begin{notation}
Let $F$ be a  Klein foam with boundary properly embedded in $\BB^4$. 
Recall that we have a Klein cover $\pi:W_F \rightarrow W_F/\kg \simeq \BB^4$. 
For every element $i$ in $\kg^*$, we denote by: 
\begin{itemize}
 \item $\widehat{F}_i$ the fixed-points surface of the diffeomorphism $i$ in $W_F$, 
\item $W_{F}/i$ the manifold $W_F/{\left< i \right>}$ (where $\left< i \right>$ is the subgroup of order two generated by $i$),
\item $\widetilde{F}_i^i$ the image of $\widehat{F}_i$ in $W_{F}/i$,
\item $\widetilde{F}_{jk}^i$ the image of $\widehat{F}_j\cup\widehat{F}_k$ 
in $W_{F}/i$ where $j$ and $k$ denote the two other elements of $\kg^*$,
\item $F_{jk}$ the image of $\widehat{F}_j\cup\widehat{F}_k$ 
in $\BB^4$.
\end{itemize} These notations are summarized below in Figure~\ref{fig:notations}.
\end{notation}
Note that $\widetilde{F}_i^i$ and $\widetilde{F}_{jk}^i$ are properly embedded surfaces in $W_{F}/i$
and that $F_{jk}$ is a properly embedded surface in $\BB^4$: 
 in fact this is the union of the facets of $F$ colored by $j$ and $k$.
Moreover the  Klein cover $\pi$ is the composition of two double
branched covers
$W_F \rightarrow W_{F}/i \rightarrow \BB^4$. 
The last one can be seen as the double branched cover of $\BB^4$ along $F_{jk}$ 
and the first one as a double branched cover of $W_{F}/i$ along $\widetilde{F}_i^i$.
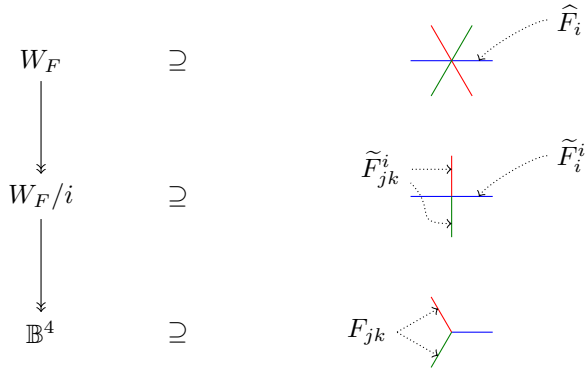
\begin{figure}[ht]
  \centering
\begin{tikzpicture}[scale =1.8]
    \node (B4) at (0,0) {$\BB^4$};
    \node (TW2) at (0,1) {$W_{F}/i$};
    \node (HW) at (0,2) {$W_F$};
    \node (B42) at (1,0) {$\supseteq$};
    \node (HW2) at (1,2) {$\supseteq$};
    \node (TW22) at (1,1) {$\supseteq$};
    \draw[->>] (HW) -- (TW2);
    \draw[->>] (TW2) -- (B4);
    \begin{scope}[xshift =3cm]
      \begin{scope}[yshift= 2cm]
        \draw[blue] (0:0.3) -- (180:0.3);
        \draw[red] (120:0.3) -- (-60:0.3);
        \draw[green!50!black] (60:0.3) -- (240:0.3);
        \draw[<-, densely dotted] (0:0.2) .. controls +(60:0.1) and +(180:0.1).. +(0.5,0.3) node[right] {$\widehat{F_i}$};
      \end{scope}
      \begin{scope}[yshift= 1cm]
        \draw[blue] (0:0.3) -- (180:0.3);
        \draw[red] (90:0.3) -- (60:0);
        \draw[green!50!black] (270:0.3) -- (240:0);
        \draw[<-, densely dotted] (0:0.2) .. controls +(60:0.1) and +(180:0.1).. +(0.5,0.3) node[right] {$\widetilde{F}^i_i$};
        \draw[<-, densely dotted] (90:0.2)-- (-0.3, 0.2) node[left] {$\widetilde{F}^i_{jk}$};
        \draw[<-, densely dotted] (-90:0.2) ..controls +(-0.3,0) and +(0.2, -0.2).. (-0.3, 0.1);
      \end{scope}
      \begin{scope}[yshift= 0cm]
        \draw[blue] (0:0.3) -- (0,0);
        \draw[red] (120:0.3) -- (0,0);
        \draw[green!50!black] (240:0.3) -- (0,0);        
        \draw[<-, densely dotted] (120:0.2)-- (-0.4, 0) node[left] {${F}_{jk}$};
        \draw[<-, densely dotted] (240:0.2) -- (-0.4, 0);
      \end{scope}
    \end{scope}
\end{tikzpicture}  
  \caption{Decomposition of a Klein cover in two double branched covers}
  \label{fig:notations}
\end{figure}
Denoting by $a$, $b$ and $c$ the three elements of $\kg^*$, 
the following diagram explains the three decompositions of the Klein cover 
of $\BB^4$ on $F$ in double branched covers:
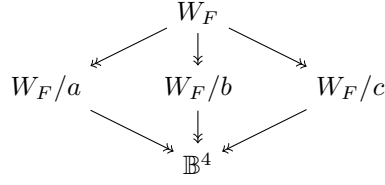
\[
\begin{tikzpicture}
    \node (HW) at (0,2) {$W_F$};
    \node (TW1) at (-2,1) {$W_{F}/\aa$};
    \node (TW2) at (0,1) {$W_{F}/\bb$};
    \node (TW3) at (2,1) {$W_{F}/\cc$};
    \node (B4) at (0,0) {$\BB^4$};
    \draw[->>] (HW) -- (TW1);
    \draw[->>] (HW) -- (TW2);
    \draw[->>] (HW) -- (TW3);
    \draw[->>] (TW1) -- (B4);
    \draw[->>] (TW2) -- (B4);
    \draw[->>] (TW3) -- (B4);
\end{tikzpicture}
\]
Moreover we will use the same system of notations:
\begin{itemize}
 \item for a Klein foam without boundary $E$ embedded in $\SS^4$,
 \item for a Klein graph $\Gamma$ embedded in $\SS^3$.
\end{itemize}
In this last case, the Klein cover $M_\Gamma$ has the following decomposition in 
double branched covers:
\[
\begin{tikzpicture}
    \node (HW) at (0,2) {$M_\Gamma$};
    \node (TW1) at (-2,1) {$M_{\Gamma}/\aa$};
    \node (TW2) at (0,1) {$M_{\Gamma}/\bb$};
    \node (TW3) at (2,1) {$M_{\Gamma}/\cc$};
    \node (B4) at (0,0) {$\SS^3$};
    \draw[->>] (HW) -- (TW1);
    \draw[->>] (HW) -- (TW2);
    \draw[->>] (HW) -- (TW3);
    \draw[->>] (TW1) -- (B4);
    \draw[->>] (TW2) -- (B4);
    \draw[->>] (TW3) -- (B4);
\end{tikzpicture}
\]

\begin{rmk}
\label{knotsinagraph}
\begin{enumerate}
\item The sub-graphs $\Gamma_{\aa\bb}$, $\Gamma_{\bb\cc}$ and $\Gamma_{\cc\aa}$ are links in $\SS^3$. For $\{i,j,k\} = \{\aa,\bb,\cc\}$,  $\widetilde{\Gamma}_{jk}^i$ and $\widetilde{\Gamma}_{i}^i$ are links in $M_\Gamma/i$ and $\widehat{\Gamma}_{\aa}$, $\widehat{\Gamma}_{\bb}$ and $\widehat{\Gamma}_{\cc}$ are links in $M_{\Gamma}$.
\item If $\Gamma$ is 3-Hamiltonian, $\Gamma_{\aa\bb}$, $\Gamma_{\bb\cc}$ and $\Gamma_{\cc\aa}$ are knots. In this case $M_{\Gamma}/\aa$, $M_{\Gamma}/\bb$ and $M_{\Gamma}/\cc$ are rational homology spheres whose first homology group has no $2$-torsion (see for instance \cite[p. 213]{MR1277811}).  
\end{enumerate}
\end{rmk}
Finally, we define normal Euler numbers for Klein foams as follows.
\begin{dfn}
  \label{dfn:eulernumberfoams}
  Let $F$ be a Klein foam with boundary in $(\BB^4, \SS^3)$. The \emph{weak normal Euler number} $e(F)$ of $F$ is given by the following formula:
\[e(F) = e(F_{\aa \bb}) + e(F_{\bb \cc}) + e(F_{\cc \aa}). \]
If $\partial F$ is 3-Hamiltonian, we define the \emph{strong normal Euler number} $\widetilde{e}(F)$ of $F$ by:
\[\widetilde{e}(F)=  e(\widetilde{F}_{\aa}^{\aa}) +e(\widetilde{F}_{\bb}^{\bb})+e(\widetilde{F}_{\cc}^{\cc}). \]  
\end{dfn}

We are now able to state the main theorem and to define our invariants. 
Recall that for a knotted Klein graph $\Gamma$ in $\SS^3$ 
there exists a spanning foam for $\Gamma$ in $\BB^4$ (Proposition \ref{prop:spanningfoamexists}).

\begin{thm}
  \label{thm:main}

Let $\Gamma$ be a knotted 3-Hamiltonian Klein graph in $\SS^3$. 
Let $F$ be a spanning foam for $\Gamma$ in $\BB^4$. We denote the signature of the 4-manifold $W_F$ by $\sigma(W_F)$. 
\begin{itemize}
\item The integer $\sigma(\Gamma):= \sigma(W_F) + \frac12 e(F)$,
\item the rational  $\widetilde{\sigma}(\Gamma) := \sigma(W_F) +\frac12 \widetilde{e}(F)$,
\item the rational $\delta(\Gamma) := \frac12 \widetilde{e}(F) - \frac12 e(F)$
\item and the rationals:
  \begin{align*}
\delta_{\aa \bb}(\Gamma)
 &:=\frac14e(\widetilde{F}^{\aa}_{\aa})+\frac14e(\widetilde{F}^{\bb}_{\bb}) -\frac12e(F_{\aa \bb}),\\
\delta_{\bb \cc}(\Gamma)
 &:=\frac14e(\widetilde{F}^{\bb}_{\bb})+\frac14e(\widetilde{F}^{\aa}_{\aa}) -\frac12e(F_{\bb \cc}),\\    
\delta_{\cc \aa}(\Gamma) 
 &:=\frac14e(\widetilde{F}^{\cc}_{\cc})+\frac14e(\widetilde{F}^{\aa}_{\aa}) -\frac12e(F_{\cc \aa})    
  \end{align*}
\end{itemize}
only depend on the knotted Klein graph $\Gamma$. These quantities are called \emph{signature invariants of $\Gamma$}. 
\end{thm}

Of course, we have the following relations between these invariants:
\begin{align*}
  \delta(\Gamma) = \widetilde{\sigma}(\Gamma) -\sigma(\Gamma) \quad \textrm{and} \quad 
  \delta(\Gamma) = \delta_{\aa \bb}(\Gamma) + \delta_{\bb  \cc}(\Gamma) + \delta_{\cc \aa}(\Gamma).
\end{align*}
Before relating our invariants with signatures of knots and links, let us recall the following definition due to Gilmer~\cite{MR1145914, MR1238876}:

\begin{dfn}
  \label{dfn:xiinM}
  Let $M$ be a rational homology sphere with a first homology group of
  odd order and $L$ a (unoriented) link in $M$. We can find an
  oriented four dimensional manifold $W$ and a surface $F$ such that
  the pair $(W,F)$ bounds a $r$ copies of $(M,L)$ for a positive
  integer $r$. Then the \emph{signature $\xi(L)$} of $L$ is defined by
  the following formula:
  \[
\xi(L) = \frac{1}{r}\left( \sigma(W_F) - 2 \sigma(W) +\frac{1}{2} e(F)  \right)
\]
where $W_F$ denotes the double branched cover of $W$ along
$F$.
\end{dfn}

\begin{rmk}
  \label{rmk:xi}
  This definition is given in \cite{MR1145914} for an oriented link
  and in a more general setting where $\xi$ depends on a choice of a
  branched covering of $M$ along $L$. With our assumptions, such a
  choice is unique and $\xi$ does not depend on the orientation of $L$
  (see \cite[p. 295]{MR1238876}). If $M$ is
  $\SS^3$ this definition agrees with the signature of an unoriented
  link given by Murasugi~\cite{Murasugi}.
\end{rmk}

\begin{prop}
  \label{prop:weaksignaturebicolor} 
  Let $\Gamma$ be a knotted 3-Hamiltonian Klein graph in $\SS^3$. The
  following relations holds:

    \begin{align*}
      \sigma(\Gamma) &= \sigma(\Gamma_{\aa \bb}) + \sigma(\Gamma_{\bb
      \cc}) +\sigma(\Gamma_{\cc\aa}), \\
      \widetilde{\sigma}(\Gamma) &=
      \xi(\widetilde{\Gamma}^{\aa}_{\aa}) +
      \xi(\widetilde{\Gamma}^{\bb}_{\bb})
      +\xi(\widetilde{\Gamma}^{\cc}_{\cc}), \\
      \delta_{ij}(\Gamma)
  &=\frac12\xi(\widetilde{\Gamma}^{i}_i)+\frac12\xi(\widetilde{\Gamma}^{j}_j)
      -\sigma (\Gamma_{ij}) \quad \textrm{for all $i\neq j$ in
      $\kg^*$}
    \end{align*}
    Note that $\widetilde{\Gamma}^{\aa}_{\aa}$,
    $\widetilde{\Gamma}^{\bb}_{\bb}$ and $\widetilde{\Gamma}^{\cc}_{\cc}$
    are \emph{links} in rational homology spheres.
    In these formulas $\sigma(K)$ denotes the signature of a knot $K$ in $\SS^3$.
\end{prop}

The proofs of Theorem~\ref{thm:main} and Proposition~\ref{prop:weaksignaturebicolor} are postponed to Section~\ref{sec:proof-main-theorem}.

\begin{rmk}\label{rmk:homologydoublecover}
\begin{enumerate}
\item 
In some cases, this proposition enables to compute signature of the knots $\widetilde{\Gamma_i^i}$ by taking advantage of some symmetries. See Section~\ref{sec:an-example} for an example.
\item The invariants $\sigma$, $\widetilde{\sigma}$ and $\delta_{ij}$ are additive with respect to connected sum along a vertex (see Remark~\ref{rmk:graphs}(\ref{item:connsum})). For $\sigma$, it follows from the additivity of the signature of knots with respect to connected sum. For the invariants $\delta_{ij}$, it follows from the additivity of the normal Euler numbers of surfaces. From this, one easily deduces the additivity of $\widetilde{\sigma}$.
\end{enumerate}
\end{rmk}

The proofs of the results of this section require some properties on normal Euler numbers and coverings which are described in the next section.

\subsection{Normal Euler numbers and double branched covers}
\label{sec:normal-euler-number}
In this part, we explain how normal Euler numbers behave when considering double branched coverings. We consider three different situations. The first two appear in  Lemma \ref{lem:kauffman} and are classical. The third one is quite specific to our situation. It is given in Lemma~\ref{lem:ENalmosttransverse}.

\begin{lem}
\label{lem:kauffman}
\begin{enumerate}
\item\label{it:kauf1} Let $\Sigma_1$ be a surface embedded in a $4$-manifold $W$ and
  $\widetilde{W}$ be the double branched cover of $W$ along
  $\Sigma_1$. We define $\widetilde{\Sigma_1}$ to be the pre-image of
  $\Sigma_1$ in $\widetilde{W}$. We have:
  \[e(\Sigma_1) = 2e(\widetilde{\Sigma_1}).\] 
  \item\label{it:kauf2} Suppose that another
  surface $\Sigma_2$ intersects $\Sigma_1$ transversely. We denote by
  $\widetilde{\Sigma_2}$ the pre-image of $\Sigma_2$ in
  $\widetilde{W}$. We have:
  \[e(\Sigma_2) = \frac{1}{2}e(\widetilde{\Sigma_2}).\]
\end{enumerate}
\end{lem}

 \begin{proof}
(\ref{it:kauf1}) We denote by $\pi: \widetilde{W} \to W$ the canonical projection.
We choose a section $\widetilde{s}$ of the normal bundle of $\widetilde{\Sigma_1}$ transverse to the trivial section $\widetilde{s}_0$. The section $\widetilde{s}$ induces via $\pi$ a section $s$ of the normal bundle of $\Sigma$. This section $s$ is not transverse to the normal section $s_0$ of $\Sigma$. 
However there are only finitely many intersection points. In a neighborhood of every intersection point, the double branched cover can be written in an appropriate chart as:
\[ 
  \begin{array}{lrcl}
    \pi \colon& \BB^4:=  \{(z_1, z_2)\in \CC^2|\, |z_1| \leq 1 |z_2|\leq 1\} & \to  & \BB^4 \\
        & (z_1, z_2) &\mapsto& (z_1, z_2^2) 
  \end{array}
\]
with $\widetilde{\Sigma_1} = \{(z_1, 0)|\, |z_1|\leq 1 \}$ and\footnote{When working in a chart, we identify a  section of the normal bundle of a surface with its image in the ambient 4-manifold by an appropriate exponential map. In particular the null section is identified with the surface itself.} $\widetilde{s}(\widetilde{\Sigma_1})= \{(z_1, z_1)| \, |z_1| \leq 1\}$. Locally the intersection number of $\widetilde{\Sigma_1}$ with $\widetilde{s}$ is $\pm1$, the sign depending on the orientation of $\BB^4$. With this setting, $\Sigma_1= \{(z_1, 0)|\, |z_1|\leq 1 \}$ and $s=\{(z_1, z_1^2)| \, |z_1| \leq 1\}$. The local intersection number of $\Sigma_1$ with $s$ is $\pm2$, with the same sign. We obtain the result by summing over all intersection points. 

(\ref{it:kauf2}) We consider $s$ a section of the normal bundle of $\Sigma_2$ transverse to the null section. We may assume that $\Sigma_2\cap s(\Sigma_2)$ and $\Sigma_1$ are disjoint. The section $s$ induces a section $\widetilde{s}$ of the normal bundle of  $\widetilde{\Sigma_2}$. The section $\widetilde{s}$ is transverse to the null section. Moreover the map $\pi$ induces a 2 to 1 correspondence between  $\widetilde{\Sigma_2}\cap \widetilde{s(\Sigma_2)}$ and $\Sigma_2\cap s(\Sigma_2)$ which respects the intersection signs.
\end{proof}

\begin{dfn}
  \label{dfn:quasi-transverse}
  Let $\Sigma$ and $\Sigma'$ be two surfaces embedded in a 4-manifold $W$. We say that $\Sigma$ and $\Sigma'$ intersect \emph{quasi-transversely} if for all $x$ in $\Sigma\cap \Sigma'$, there exists a neighborhood $U$ of $x$ in $W$ such that:
  \[
(U, \Sigma \cap U, \Sigma'\cap U, x) \simeq (\RR^4, \RR^2\times \{(0,0)\}, \RR\times \{0\} \times \RR \times \{0\}, \{(0,0,0,0)\}).
\]
\end{dfn}

\begin{lem}
  \label{lem:ENalmosttransverse} 
 Let $\Sigma$ be a surface embedded in an oriented $4$-manifold $W$. Let $\widetilde{W}$ a double branched cover of $W$ along $\Sigma$ and  $\widetilde{\Sigma}$ the pre-image of $\Sigma$ in $\widetilde{W}$.
Suppose that a surface $\Sigma'$ in $W$  intersects $\Sigma$ quasi-transversely. Let $\widetilde{\Sigma'}$ be the pre-image of $\Sigma'$ in $\widetilde{W}$. We have:
\[
e(\Sigma') = \frac12 e(\widetilde{\Sigma'})
\]
Note that the surface $\widetilde{\Sigma'}$ is not embedded. 
\end{lem}

\begin{proof}

As a double branched cover, the manifold $\widetilde{W}$ is naturally endowed with an involution. We denote it by $\tau$. 
The intersection of $\Sigma$ and $\Sigma'$ is a finite collection $\mathcal{C}$ of circles. 

We construct a $\tau$-invariant section  $\widetilde{s}'$ of the normal bundle of $\widetilde{\Sigma'}$ in a tubular neighborhood of $\mathcal{C}$. In order to do this, we fix a base point $x_C$ on each circle $C$ and define the section locally. This can be extended along the circle. 

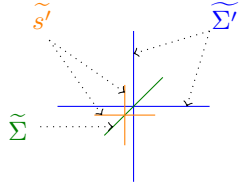
\[
  \begin{tikzpicture}
    \begin{scope}
  \draw[blue] (-1,0,0) -- (1,0,0);
  \draw[blue] (0,-1,0) -- (0,1,0); 
  \draw[green!50!black] (0,0,-1) -- (0,0,1);
  \draw[orange] (-0.4, 0, 0.3 ) -- (0.4, 0, 0.3 );
  \draw[orange] (0, -0.4, 0.3 ) -- ( 0, 0.4, 0.3 );
  \draw[thin, dotted, <-] (0.7,0,0) -- (1,1,0);
  \draw[thin, dotted, <-] (0,0.7,0) -- (1,1,0);
  \node[blue] at (1.2,1.2,0) {$\widetilde{\Sigma'}$};
  \draw[thin, dotted, <-] (0,0,0.7) -- (-1,0,0.7) node[left, green!50!black] {$\widetilde{\Sigma}$};
  \draw[thin, dotted, <-] (-0.3,0,0.3) -- (-1,1,0.3);
  \draw[thin, dotted, <-] (0,0.3,0.3) -- (-1,1,0.3);
  \node[orange] at (-1.2, 1.2) {$\widetilde{s'}$};

\end{scope}
  \end{tikzpicture}
\]

When reaching the base point $x_C$, we may face a gluing problem. In that case we complete the section as described in the picture:

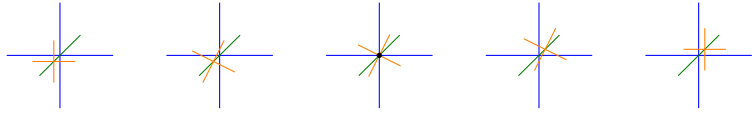
\[
  \begin{tikzpicture}[scale=0.7]
    \begin{scope}
  \draw[blue] (-1,0,0) -- (1,0,0);
  \draw[blue] (0,-1,0) -- (0,1,0); 
  \draw[green!50!black] (0,0,-1) -- (0,0,1);
  \draw[orange] (-0.4, 0, 0.3 ) -- (0.4, 0, 0.3 );
  \draw[orange] (0, -0.4, 0.3 ) -- ( 0, 0.4, 0.3 );
\end{scope}

\begin{scope}[xshift=3cm]
  \draw[blue] (-1,0,0) -- (1,0,0);
  \draw[blue] (0,-1,0) -- (0,1,0); 
  \draw[green!50!black] (0,0,-1) -- (0,0,1);
  \draw[orange] (-0.4, 0.2, 0.3 ) -- (0.4, -0.2, 0.3 );
  \draw[orange] (-0.2, -0.4, 0.3 ) -- ( 0.2, 0.4, 0.3 );
\end{scope}

\begin{scope}[xshift=6cm]
  \draw[blue] (-1,0,0) -- (1,0,0);
  \draw[blue] (0,-1,0) -- (0,1,0); 
  \draw[green!50!black] (0,0,-1) -- (0,0,1);
  \draw[orange] (-0.4, 0.2, 0 ) -- (0.4, -0.2, 0 );
  \draw[orange] (-0.2, -0.4, 0 ) -- ( 0.2, 0.4, 0 );
  \fill (0,0,0)  circle (0.5mm);
\end{scope}

\begin{scope}[xshift=9cm]
  \draw[blue] (-1,0,0) -- (1,0,0);
  \draw[blue] (0,-1,0) -- (0,1,0); 
  \draw[green!50!black] (0,0,-1) -- (0,0,1);
  \draw[orange] (-0.4, 0.2, -0.3 ) -- (0.4, -0.2, -0.3 );
  \draw[orange] (-0.2, -0.4, -0.3 ) -- ( 0.2, 0.4, -0.3 );
\end{scope}

\begin{scope}[xshift=12cm]
  \draw[blue] (-1,0,0) -- (1,0,0);
  \draw[blue] (0,-1,0) -- (0,1,0); 
  \draw[green!50!black] (0,0,-1) -- (0,0,1);
  \draw[orange] (-0.4, 0, -0.3 ) -- (0.4, 0, -0.3 );
  \draw[orange] (0, -0.4, -0.3 ) -- ( 0, 0.4, -0.3 );
\end{scope} 
  \end{tikzpicture}
\]

This yields a pair of transverse intersection points of the $\widetilde{s}'$ with the null section of the normal bundle.

The image $s'$ of this (incomplete) section gives an (incomplete) transverse section of the normal bundle of $\Sigma'$ transverse to the trivial section. The image of the neighborhoods of the intersection points is given by:

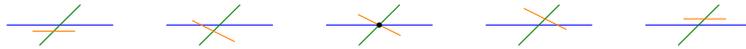
\[
  \begin{tikzpicture}[scale=0.7]
    \begin{scope}
  \draw[blue] (-1,0,0) -- (1,0,0);
  \draw[green!50!black] (0,0,-1) -- (0,0,1);
  \draw[orange] (-0.4, 0, 0.3 ) -- (0.4, 0, 0.3 );

\end{scope}

\begin{scope}[xshift=3cm]
  \draw[blue] (-1,0,0) -- (1,0,0);
  \draw[green!50!black] (0,0,-1) -- (0,0,1);
  \draw[orange] (-0.4, 0.2, 0.3 ) -- (0.4, -0.2, 0.3 );

\end{scope}

\begin{scope}[xshift=6cm]
  \draw[blue] (-1,0,0) -- (1,0,0);
  \draw[green!50!black] (0,0,-1) -- (0,0,1);
  \draw[orange] (-0.4, 0.2, 0 ) -- (0.4, -0.2, 0 );

  \fill (0,0,0)  circle (0.5mm);
\end{scope}

\begin{scope}[xshift=9cm]
  \draw[blue] (-1,0,0) -- (1,0,0);
  \draw[green!50!black] (0,0,-1) -- (0,0,1);
  \draw[orange] (-0.4, 0.2, -0.3 ) -- (0.4, -0.2, -0.3 );

\end{scope}

\begin{scope}[xshift=12cm]
  \draw[blue] (-1,0,0) -- (1,0,0);
  \draw[green!50!black] (0,0,-1) -- (0,0,1);
  \draw[orange] (-0.4, 0, -0.3 ) -- (0.4, 0, -0.3 );

\end{scope} 
  \end{tikzpicture}
\]
 
Hence, to each pair of intersection points in $\widetilde{W'}$ corresponds a single intersection point of the same sign. We extend the section $s'$ into a complete section of the normal bundle of $\Sigma'$. Its pre-image extend $\widetilde{s}'$ into a complete section of the normal bundle of  $\widetilde{W'}$. What happens far from the circles is  given by part~(\ref{it:kauf2}) of Lemma~\ref{lem:kauffman}.
\end{proof}

\subsection{Proof of the main theorem}
\label{sec:proof-main-theorem}

The principle of the proof is inspired by \cite{GL78,GL79}  which use the 4-dimensional point of view 
of Kauffman and Taylor \cite{KT76} on signature of knots.
It uses the $G$-signature theorem by Atiyah and Singer~\cite{MR0236952}.

Following \cite{GSign} let us recall the definition of signature in the context of 4-dimensional manifolds. 
Let $W$ be a compact oriented $4$-manifold (with or without boundary).
The intersection form of $W$ induces an hermitian form
$\varphi$ on $H_2(W,\CC)$.   
Take a finite group $G$ which
acts on $W$ by orientation-preserving diffeomorphisms. Then for every $g$ in $G$, $g_*$ preserves the form 
$\varphi$ on $H_2(W,\CC)$ and we may choose a $G$-invariant $\varphi$-orthogonal decomposition 
$H_2(W,\CC)=H^+ \oplus H^- \oplus H^0$ such that $\varphi$ is positive definite on $H^+$, negative definite on $H^-$ 
and zero on $H^0$.
The \emph{signature of $g$} is defined by:
\[
\sigma(W,g) := \sigma(H_2(W),g_*) = \mathrm{tr}({g_*}_{|H^+}) - \mathrm{tr}({g_*}_{|H^-}),\] 
where $\mathrm{tr}$ denotes the trace of linear endomorphisms.

\begin{rmk}
  \label{rmk:signid}
  The signature $\sigma(W$) of a manifold $W$ is the signature of the identity of $W$. 
\end{rmk}

The following formula connects the sum of the signature of the elements of $G$ and the signature of the 
quotient manifold $W/G$ and is known in the literature as a {standard transfer argument}.

\begin{prop}\label{prop:transfert}
 Let $W$ be a compact oriented $4$-manifold endowed with an action of a finite group $G$ by orientation-preserving diffeomorphisms.
 Then we have: 
 \[\sum_{g \in G}\sigma(W,g) = |G|\sigma(W/G)\]
\end{prop}
\begin{proof}[Sketch of the proof]
Let $\pi$ be the projection $\pi: W \longrightarrow W/G$.
Then there exists cellular decompositions of $W$ and $W/G$ compatible with the action of $G$ such that $\pi$ 
is cellular. Moreover, at the level of chains one can define a map which assigns to each 2-chain of $W/G$ the sum 
of its pre-images by $\pi_*$. This induces a transfer map $t:H_2(W/G) \longrightarrow H_2(W)^G \hookrightarrow H_2(W)$ 
respecting intersections forms in the sense that $\varphi_W ( tx, ty) = |G| \varphi_{W/G} ( x, y)$.
Moreover we have 
$\pi_* \circ t = |G| \id_{H_2(W/G)}$ and $t \circ \pi_* = \sum_{g \in G} g_*$ on $H_2(W)$.
It follows that:
\begin{align*}
\sum_{g \in G}\sigma(W,g)  & = \sigma (H_2(W), t \circ \pi_*)\\
& = \sigma (H_2(W)^G, t \circ \pi_*)\\
& = |G|\sigma (H_2(W)^G, \mathrm{id} )\\  
& = |G| \sigma (W/G).
\end{align*}
\end{proof}
The $G$-signature theorem for $4$-manifolds \cite[Theorem 2]{GSign} takes an especially simple form for involutions:
\begin{thm}
  \label{thm:Gsignature}
Let $W$ be an oriented closed $4$-manifold and $\tau$ be an orientation-preserving smooth involution. 
Then the set of fixed points of $\tau$ consists of a surface $F_\tau$ and  some isolated points and we have:
\[  \sigma(W,\tau)=  e(F_\tau). \]
\end{thm}

\begin{rmk}
  If $\tau$ is not an involution (but still has finite order), the statement is a bit more complicated and there are some contribution coming from the isolated fixed points.

\end{rmk}

\begin{proof}[Proof of the Theorem~\ref{thm:main}]

We first prove the invariance of $\sigma$, and $\widetilde{\sigma}$. This implies the invariance of $\delta$. 
The invariance of the $\delta_{ij}$'s is postponed to the proof of Proposition~\ref{prop:weaksignaturebicolor}.

Let us  consider two spanning foams  $(\BB^4_1,F_1)$ and $(\BB^4_2,F_2)$ for $(\SS^3,\Gamma)$ 
(the indices of the $\BB^4$ are simply labels with no geometrical meaning). 
Then we can consider the gluing $(\SS^4,E) = (\BB^4_1,F_1) \cup_{(\SS^3,\Gamma)} (-\BB^4_2,F_2)$. 
Lifting  to Klein covers we get the gluing $V_E = W_{F_1} \cup_{M_\Gamma} (-W_{F_2})$.
Novikov additivity \cite[Theorem 5.3]{MR1001966}  gives: 
\begin{align}\label{proof:eq1}
   \sigma (V_E) = \sigma({W}_{F_1}) - \sigma({W}_{F_2}). 
  \end{align}

The action of the Klein group $\kg$ on $V_E$ defines a collection of signatures which satisfies, 
thanks to Proposition \ref{prop:transfert}, the relation:
\[\sigma (V_E) + \sum_{g \in \kg^*}\sigma(V_E,g) = 4\sigma(\SS^4)=0.\]
Together with Theorem~\ref{thm:Gsignature} this gives:
\begin{align}\label{proof:eq2}
-\sigma (V_E) =& \sigma(V_E,\aa) +\sigma(V_E,\bb) + \sigma(V_E,\cc) 
 =e(\widehat{E}_{\aa}) + e(\widehat{E}_{\bb}) + e(\widehat{E}_{\cc}).
 \end{align}
 In the diagram 
 
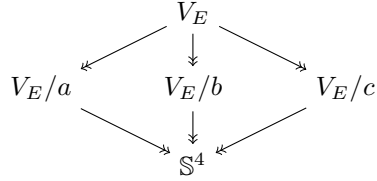
\[
\begin{tikzpicture}
    \node (HW) at (0,2) {$V_E$};
    \node (TW1) at (-2,1) {$V_{E}/\aa$};
    \node (TW2) at (0,1) {$V_{E}/\bb$};
    \node (TW3) at (2,1) {$V_{E}/\cc$};
    \node (B4) at (0,0) {$\SS^4$};
    \draw[->>] (HW) -- (TW1);
    \draw[->>] (HW) -- (TW2);
    \draw[->>] (HW) -- (TW3);
    \draw[->>] (TW1) -- (B4);
    \draw[->>] (TW2) -- (B4);
    \draw[->>] (TW3) -- (B4);
\end{tikzpicture}
\]
the three top arrows are double branched covers along $\widetilde{E}^{\aa}_{\aa}$, $\widetilde{E}^{\bb}_{\bb}$ and
$\widetilde{E}^{\cc}_{\cc}$ respectively. 
So applying first part of Lemma~\ref{lem:kauffman}, we get:
\begin{align*}
-\sigma (V_E) = & \frac{1}{2}e(\widetilde{E}^{\aa}_{\aa}) + \frac{1}{2}e(\widetilde{E}^{\bb}_{\bb}) 
+ \frac{1}{2}e(\widetilde{E}^{\cc}_{\cc}).
\end{align*}
Note that in the gluing $V_{E}/\aa = (W_{F_1}/\aa) \cup_{M_{\Gamma}/\aa} (-W_{F_2}/\aa)$ 
the 3-manifold $M_{\Gamma}/\aa$ is a rational 
homology sphere (see Remark~\ref{rmk:homologydoublecover}) so that thanks to Proposition~\ref{prop:en-change-orientation-and-glue} and Definition~\ref{dfn:eulernumberfoams} we obtain:
\begin{align*}
-\sigma (V_E) = & \frac{1}{2}\left(e(\widetilde{(F_1)}^{\aa}_{\aa}) - e( \widetilde{(F_2)}^{\aa}_{\aa} ) 
+ e(\widetilde{(F_1)}^{\bb}_{\bb}) - e(\widetilde{(F_2)}^{\bb}_{\bb} ) 
+ e(\widetilde{(F_1)}^{\cc}_{\cc})-e( \widetilde{(F_2)}^{\cc}_{\cc} )\right)\\ 
=& \frac12\widetilde{e}(F_1) - \frac12\widetilde{e}(F_2).
\end{align*}
From this and Relation~(\ref{proof:eq1}) we deduce:
\[\sigma({W}_{F_1}) + \frac12\widetilde{e}(F_1) = \sigma({W}_{F_2}) + \frac12\widetilde{e}(F_2).
\]
This proves the invariance of $\widetilde{\sigma}$.

On the other hand, starting again from relation (\ref{proof:eq2}) we get:
\begin{align*}
-\sigma (V_E) = & \frac{1}{2}\left(e( \widehat{E}_{\aa}) + e(\widehat{E}_{\bb})\right) + \frac{1}{2}\left(e( \widehat{E}_{\bb}) + e(\widehat{E}_{\cc})\right)
+ \frac{1}{2}\left(e( \widehat{E}_{\cc}) + e(\widehat{E}_{\aa})\right)\\
= & \frac{1}{2}e( \widehat{E}_{\aa}\cup\widehat{E}_{\bb}) + \frac{1}{2}e( \widehat{E}_{\bb}\cup\widehat{E}_{\cc})
+ \frac{1}{2}e( \widehat{E}_{\cc}\cup\widehat{E}_{\aa}) \\
=  &
e(\widetilde{E}^{\cc}_{\aa, \bb}) + 
e(\widetilde{E}^{\aa}_{\bb,\cc}) + e(\widetilde{E}^{\bb}_{\cc,\aa} ) \\
= & \frac12\left( e(E_{\aa \bb}) + e(E_{\bb \cc}) + e(E_{\cc \aa})\right) \\
= & \frac12 e(F_1) - \frac12 e(F_2).
\end{align*}
The third and fourth equalities follows from Lemmas~\ref{lem:ENalmosttransverse} and \ref{lem:kauffman}, while the last one holds because of Proposition~\ref{prop:en-change-orientation-and-glue}.    
Again from this and relation (\ref{proof:eq1}) we deduce:
\[
\sigma({W}_{F_1}) +\frac12e(F_1) = \sigma({W}_{F_2}) + \frac12e(F_1)
\]
This proves the invariance of ${\sigma}$.

\end{proof}

Before proving Proposition~\ref{prop:weaksignaturebicolor}, we need
the following lemma.

 \begin{lem}\label{lem:sigmasigma} 
Let $F$ be foam with boundary in $\mathbb{B}^4$, then we have:
\[
\sigma (W_F) = \sigma(W_F/\aa) +\sigma(W_F/\bb) +
\sigma(W_F/\cc).
\]
 \end{lem}
\begin{proof} 
Using Proposition~\ref{prop:transfert} for the manifold
$W_F$ endowed with the action of the Klein group $\kg$ we get:
\[
\sigma(W_F)+\sigma(W_F,\aa)+\sigma(W_F,\bb)+\sigma(W_F,\cc)=4\sigma(\BB^4)=0.
\]
Now the same proposition applied to the action of the involution $\aa$
on $W_F$ gives:
\[
\sigma(W_F)+\sigma(W_F,\aa)=2\sigma(W_F/\aa)
\]
and analogous formulas for involutions $\bb$ and $\cc$.  The result
follows easily.
\end{proof}

\begin{proof}[Proof of Proposition~\ref{prop:weaksignaturebicolor}]

 Let  $F$ be an arbitrary spanning foam for $\Gamma$ in $\BB^4$.
The surface $F_{\aa \bb}$ is a spanning surface for the knot 
$\Gamma_{\aa \bb}$. Thanks to \cite[Theorem 2 and Corollary 5]{GL78}, we have: 
\begin{align}\label{proof:eq3} 
 \sigma(\Gamma_{\aa \bb}) = \sigma (W_F/\cc) + \frac12e(F_{\aa \bb}).
\end{align}

We obviously have analogous formulas for $\sigma(\Gamma_{\bb \cc})$ 
and $\sigma(\Gamma_{\cc \aa})$. This gives:
\begin{align*}
 \sigma(\Gamma_{\aa \bb})+ \sigma(\Gamma_{\bb \cc})+ \sigma(\Gamma_{\cc \aa})
 = & \sigma(W_F/\cc) +\sigma(W_F/\aa) + \sigma(W_F/\bb)
 \\ & \quad +\frac12\left( e(F_{\aa \bb}) + e(F_{\bb \cc}) + e(F_{\cc \aa})\right)\\
 = & \sigma (W_F) + \frac12e(F)&  
\\
 = & \sigma(\Gamma),
\end{align*}
where the second equality comes from Lemma~\ref{lem:sigmasigma}. This is exactly the first part of Proposition~\ref{prop:weaksignaturebicolor}.

The second part is in some sense a generalization. 
Here $\widetilde{\Gamma}^{\aa}_{\aa}$ is a link in $M_\Gamma/a$ and $\widetilde{F}_{\aa}^{\aa}$ is a spanning surface for this link 
(living in $W_F/\aa$). Recall that $M_\Gamma/a$ is a rational homology whose first homology group has odd order. Hence Definition~\ref{dfn:xiinM} gives: 
\begin{align}\label{proof:eq4}
\xi(\tilde{\Gamma}_\aa^\aa) = 
\sigma(W_F) -2 \sigma(W_F/\aa)  +\frac12 e(\widetilde{F}^\aa _\aa).
\end{align}
Doing the sum with the analogous formulas for $\xi(\tilde{\Gamma}_\bb^\bb)$ and
$\xi(\tilde{\Gamma}_\cc^\cc)$ we get:
\begin{align*}
 \xi(\tilde{\Gamma}_\aa^\aa) + \xi(\tilde{\Gamma}_\bb^\bb) + \xi(\tilde{\Gamma}_\cc^\cc)
 = & 3\sigma (W_F) -2 \left(\sigma(W_F/\aa) +\sigma(W_F/\bb) + \sigma(W_F/\cc)\right)
 \\ & \quad+\frac12\left(e(\widetilde{F}^\aa _\aa)+e(\widetilde{F}^\bb _\bb)+e(\widetilde{F}^\cc _\cc)\right)\\
 = & \sigma (W_F) + \frac12\widetilde{e}(F) \\
 = & \widetilde{\sigma}(\Gamma),
\end{align*} 
where  the second equality follows from Lemma~\ref{lem:sigmasigma}.

It remains to show that 
\[
 \delta_{ij}(\Gamma) =\frac12\xi(\widetilde{\Gamma}^{i}_i)+\frac12\xi(\widetilde{\Gamma}^{j}_j)
    -\sigma (\Gamma_{ij}) \quad \textrm{for all $i\neq j$ in $\kg^*$}.
\]
Note that this formula implies the invariance of the $\delta_{ij}$'s and therefore completes the proof of Theorem~\ref{thm:main}.

By symmetry it is enough to consider $i=a$ and $j=b$. We have:
\begin{align*}
&  \frac12\xi(\widetilde{\Gamma}^{\aa}_\aa)+\frac12\xi(\widetilde{\Gamma}^{\bb}_\bb)
    -\sigma (\Gamma_{\aa\bb}) \\ 
\qquad &= \frac12\left(\sigma(W_F) -2 \sigma(W_F/\aa)  +\frac12 e(\widetilde{F}^\aa _\aa)
+\sigma(W_F) -2 \sigma(W_F/\bb)  +\frac12 e(\widetilde{F}^\bb _\bb) \right) \\
& \qquad \qquad
- \left(\sigma (W_F/\cc) + \frac12e(F_{\aa \bb}). \right)  \\
\qquad &= \sigma(W_F) -(\sigma(W_F/\aa) + \sigma(W_F/\bb) + \sigma(W_F/\cc)) + \frac14e(\widetilde{F}^{\aa}_{\aa})+\frac14e(\widetilde{F}^{\bb}_{\bb}) -\frac12e(F_{\aa \bb})  \\
&= \delta_{ab}(\Gamma). 
\end{align*}
The last equality follows from Lemma~\ref{lem:sigmasigma}.
\end{proof}

\section{An example}
\label{sec:an-example}

In this section we compute our signature invariants on $\Gamma$, the Kinoshita knotted graph. We describe a spanning foam $F$ 
for $\Gamma$ by a movie given in Figure~\ref{fig:kinmov}. 
\begin{figure}[ht]
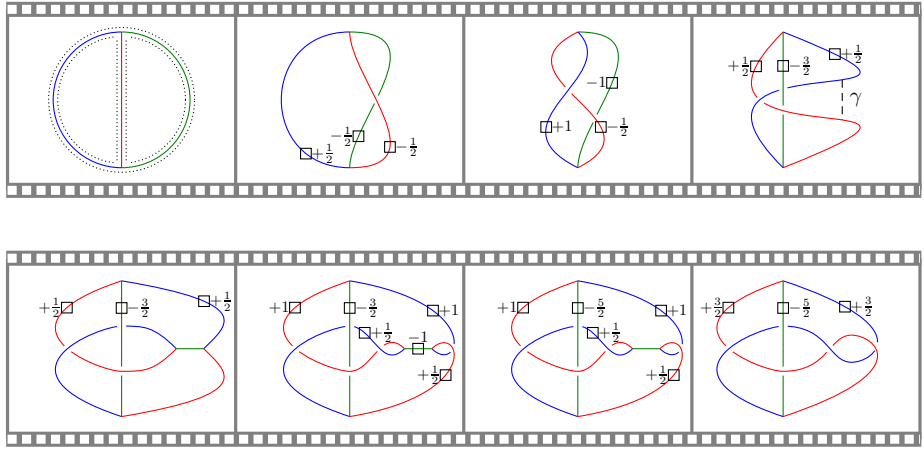

  \centering
  \tikz[scale=0.6]{\begin{scope}
\begin{scope}[scale=1]
  \draw[draw=  gray, line width = 2mm] (-0.03, 2) --  (20.03, 2); 
  \draw[draw=  gray, line width = 2mm] (-0.03,-2) --  (20.03,-2); 
  \draw[draw=  gray, very thick] (0, -2) -- +(0,4);
  \draw[draw=  gray, very thick] (5, -2) -- +(0,4);
  \draw[draw=  gray, very thick] (10,-2) -- +(0,4);
  \draw[draw=  gray, very thick] (15,-2) -- +(0,4);
  \draw[draw=  gray, very thick] (20,-2) -- +(0,4);
  \draw[draw= white, dotted, line width =1.2mm] (0,2) --  (20,2); 
  \draw[draw= white, dotted, line width = 1.2mm] (0,-2) --  (20,-2); 
\begin{scope}[xshift=2.5cm]  
  \draw[black,densely dotted] (0,0) circle (1.6cm);
  \draw[black,densely dotted] (0,0) circle (1.4cm);
  \draw[black, densely dotted]  (0.1,-1.5) -- (0.1,1.5);
  \draw[black, densely dotted]  (-0.1,-1.5) -- (-0.1,1.5);
  \fill[white] (-0.15,1.5) -- (0.15,1.5) -- (0.15,1.38) -- (-0.15,1.38);
  \fill[white] (-0.15,-1.5) -- (0.15,-1.5) -- (0.15,-1.38) -- (-0.15,-1.38);
  \draw[red]  (0,-1.5) -- (0,1.5);
  \draw[green!50!black] (0, -1.5) arc (-90:90: 1.5cm);
  \draw[blue] (0, 1.5) arc (90:270: 1.5cm);

\end{scope}
\begin{scope}[xshift=7.5cm]
   \draw[green!50!black]  (0,1.5) .. controls +(+2,0) and +(0,1) .. (0,-1.5) coordinate[pos=0.8] (G);
   \node[draw,scale=.6] at (G) {};
   \draw (G) node[left,scale=.6]{$-\frac12$};
   \fill[white] (0.6, 0) circle (1mm);
   \draw[red]  (0,-1.5) .. controls +(2,0) and +(0,-1) .. (0,1.5) coordinate[pos=0.3] (R);
   \node[draw,scale=.6] at (R) {};
   \draw (R) node[right,scale=.6]{$-\frac12$};
   \draw[blue]  (0,-1.5) .. controls +(-2,0) and +(-2,0) .. (0,1.5) coordinate[pos=0.2] (B);
   \node[draw,scale=.6] at (B) {};
   \draw (B) node[right,scale=.6]{$+\frac12$};
\end{scope}
\begin{scope}[xshift=12.5cm]
  \draw[green!50!black]  (0,1.5) .. controls +(+2,0) and +(0,1) .. (0,-1.5) coordinate[pos=0.5] (G) coordinate[pos=0.7] (GW);
  \fill[white] (GW) circle (1mm);
  \node[draw,scale=.6] at (G) {};
  \draw (G) node[left,scale=.6]{$-1$};
  \draw[red]  (0,-1.5) .. controls +(2,1) and +(-2,-1) .. (0,1.5) coordinate[pos=0.3] (R) coordinate[pos=0.57] (RW);
  \fill[white] (RW) circle (1mm);
  \node[draw,scale=.6] at (R) {};
  \draw (R) node[right,scale=.6]{$-\frac12$};
  \draw[blue]  (0,-1.5) .. controls +(-2,1) and +(1,-1) .. (0,1.5) coordinate[pos=0.3] (B);
  \node[draw,scale=.6] at (B) {};
  \draw (B) node[right,scale=.6]{$+1$};
\end{scope}
\begin{scope}[xshift=17cm]
  \draw[green!50!black]  (0,0) -- (0,-1.5) coordinate[pos=0.15] (GW);
  \fill[white] (GW) circle (1mm);
  \draw[red]  (0,-1.5) .. controls +(5,2) and +(-3,-3) .. (0,1.5) coordinate[pos=0.9] (R) coordinate[pos=0.73] (RW) coordinate[pos=0.4] (RC);
  \fill[white] (RW) circle (1mm);
  \node[draw,scale=.6] at (R) {};
  \draw (R) node[left,scale=.6]{$+\frac12$};
  \draw[blue]  (0,-1.5) .. controls +(-3,3) and +(5,-2) .. (0,1.5) coordinate[pos=0.9] (B) coordinate[pos=0.375] (BW) coordinate[pos=0.6] (BC);
  \node[draw,scale=.6] at (B) {};
  \draw (B) node[right,scale=.6]{$+\frac12$};
  \fill[white] (BW) circle (1mm);
  \draw[green!50!black]  (0,1.5) -- (0,0) coordinate[pos=0.5] (G);
  \node[draw,scale=.6] at (G) {};
  \draw (G) node[right,scale=.6]{$-\frac32$};
  \draw[densely dashed] (BC)  --  (RC) coordinate[pos=0.5] (M);
  \draw (M) node[right, scale=.8] {$\gamma$};
\end{scope}
\end{scope}
\begin{scope}[yshift=-5.5cm]
\begin{scope}[scale=1]
  \draw[draw=  gray, line width = 2mm] (-0.03, 2) --  (20.03, 2); 
  \draw[draw=  gray, line width = 2mm] (-0.03,-2) --  (20.03,-2); 
  \draw[draw=  gray, very thick] (0, -2) -- +(0,4);
  \draw[draw=  gray, very thick] (5, -2) -- +(0,4);
  \draw[draw=  gray, very thick] (10,-2) -- +(0,4);
  \draw[draw=  gray, very thick] (15,-2) -- +(0,4);
  \draw[draw=  gray, very thick] (20,-2) -- +(0,4);
  \draw[draw= white, dotted, line width =1.2mm] (0,2) --  (20,2); 
  \draw[draw= white, dotted, line width = 1.2mm] (0,-2) --  (20,-2); 
\begin{scope}[xshift=2.5cm]  
  \draw[green!50!black]  (0,0) -- (0,-1.5) coordinate[pos=0.33] (GW);
  \fill[white] (GW) circle (1mm);
  \draw[red]  (1.2,0) .. controls +(-.5,-.5) and +(.5,0) .. (0,-.5) coordinate[pos=0.42] (RWW) .. controls +(-.5,0) and +(-3,-1) .. (0,1.5) coordinate[pos=0.8] (R) coordinate[pos=0.44] (RW);
  \fill[white] (RW) circle (1mm);

  \node[draw,scale=.6] at (R) {};
  \draw (R) node[left,scale=.6]{$+\frac12$};
  \draw[blue]  (1.2,0) .. controls +(-.5,.5) and +(.5,0) ..  (0,.5)  .. controls +(-.5,0) and +(-3,1) .. (0,-1.5);
  \fill[white] (0,.5) circle (1mm);
  \draw[green!50!black]  (0,1.5) -- (0,0) coordinate[pos=0.4] (G);
  \node[draw,scale=.6] at (G) {};
  \draw (G) node[right,scale=.6]{$-\frac32$};
  \draw[red]  (0,-1.5) .. controls +(3,.5) and +(.5,-.5) .. (1.8,0);
  \draw[blue]  (0,1.5) .. controls +(3,-.5) and +(.5,.5) ..  (1.8,0) coordinate[pos=0.3] (B);
  \node[draw,scale=.6] at (B) {};
  \draw (B) node[right,scale=.6]{$+\frac12$}; 
  \draw[green!50!black]  (1.2,0) -- (1.8,0);
\end{scope}
\begin{scope}[xshift=7.5cm]
  \draw[green!50!black]  (0,0) -- (0,-1.5) coordinate[pos=0.33] (GW);
  \fill[white] (GW) circle (1mm);
  \draw[red]  (1.2,0) .. controls +(-.5,.5) and +(.5,0) ..  (0,-.5) coordinate[pos=0.43] (RWW) .. controls +(-.5,0) and +(-3,-1) .. (0,1.5) coordinate[pos=0.8] (R) coordinate[pos=0.44] (RW);
  \fill[white] (RW) circle (1mm);
  \fill[white] (RWW) circle (1mm);
  \node[draw,scale=.6] at (R) {};
  \draw (R) node[left,scale=.6]{$+1$};
  \draw[blue]  (1.2,0) .. controls +(-.5,-.5) and +(.5,0) .. (0,.5) coordinate[pos=0.75] (BB)  .. controls +(-.5,0) and +(-3,1) .. (0,-1.5);
  \fill[white] (0,.5) circle (1mm);     
  \node[draw,scale=.6] at (BB) {};
  \draw (BB) node[right,scale=.6]{$+\frac12$}; 
  \draw[green!50!black]  (0,1.5) -- (0,0) coordinate[pos=0.4] (G);
  \node[draw,scale=.6] at (G) {};
  \draw (G) node[right,scale=.6]{$-\frac32$};
  \draw[blue]  (0,1.5) .. controls +(3,-.5) and +(.6,-.6) ..  (1.8,0) coordinate[pos=0.3] (B) coordinate[pos=0.67] (BW);
  \node[draw,scale=.6] at (B) {};
  \draw (B) node[right,scale=.6]{$+1$}; 
  \fill[white] (BW) circle (1mm);
  \draw[red]  (0,-1.5) .. controls +(3,.5) and +(.6,.6) .. (1.8,0) coordinate[pos=0.4] (RR);
  \node[draw,scale=.6] at (RR) {};
  \draw (RR) node[left,scale=.6]{$+\frac12$};
  \draw[green!50!black]  (1.2,0) -- (1.8,0) coordinate[pos=0.5] (N) ;
  \node[draw,scale=.6] at (N) {};
  \draw (N) node[above,scale=.6]{$-1$}; 
\end{scope}
\begin{scope}[xshift=12.5cm]
  \draw[green!50!black]  (0,0) -- (0,-1.5) coordinate[pos=0.33] (GW);
  \fill[white] (GW) circle (1mm);
  \draw[red]  (1.2,0) .. controls +(-.5,.5) and +(.5,0) .. (0,-.5) coordinate[pos=0.43] (RWW) .. controls +(-.5,0) and +(-3,-1) .. (0,1.5) coordinate[pos=0.8] (R) coordinate[pos=0.44] (RW);
  \fill[white] (RW) circle (1mm);
  \fill[white] (RWW) circle (1mm);
  \node[draw,scale=.6] at (R) {};
  \draw (R) node[left,scale=.6]{$+1$};
  \draw[blue]  (1.2,0) .. controls +(-.5,-.5) and +(.5,0) ..  (0,.5) coordinate[pos=0.75] (BB)  .. controls +(-.5,0) and +(-3,1) .. (0,-1.5)  ;
  \fill[white] (0,.5) circle (1mm);     
  \node[draw,scale=.6] at (BB) {};
  \draw (BB) node[right,scale=.6]{$+\frac12$}; 
  \draw[green!50!black]  (0,1.5) -- (0,0) coordinate[pos=0.4] (G);
  \node[draw,scale=.6] at (G) {};
  \draw (G) node[right,scale=.6]{$-\frac52$};
  \draw[blue]  (0,1.5) .. controls +(3,-.5) and +(.6,-.6) ..  (1.8,0) coordinate[pos=0.3] (B) coordinate[pos=0.67] (BW);
  \node[draw,scale=.6] at (B) {};
  \draw (B) node[right,scale=.6]{$+1$}; 
  \fill[white] (BW) circle (1mm);
  \draw[red]  (0,-1.5) .. controls +(3,.5) and +(.6,.6) .. (1.8,0) coordinate[pos=0.4] (RR);
  \node[draw,scale=.6] at (RR) {};
  \draw (RR) node[left,scale=.6]{$+\frac12$};
  \draw[green!50!black]  (1.2,0) -- (1.8,0) coordinate[pos=0.5] (N) ;

  \end{scope}
  
\begin{scope}[xshift=17cm]
  \draw[green!50!black]  (0,0) -- (0,-1.5) coordinate[pos=0.33] (GW);
  \fill[white] (GW) circle (1mm);
  \draw[red]  (1.5,.3) .. controls +(-.5,0) and +(1,0) ..  (0,-.5) coordinate[pos=0.42] (RWW) .. controls +(-.5,0) and +(-3,-1) .. (0,1.5) coordinate[pos=0.8] (R) coordinate[pos=0.44] (RW);
  \fill[white] (RW) circle (1mm);
  \fill[white] (RWW) circle (1mm);
  \node[draw,scale=.6] at (R) {};
  \draw (R) node[left,scale=.6]{$+\frac32$};
  \draw[blue]  (0,+1.5) .. controls +(3,-1) and +(.5,0) .. (1.5,-.3) coordinate[pos=0.2] (B) coordinate[pos=0.6] (BW) .. controls +(-.5,0) and +(1,0) ..  (0,.5) .. controls +(-.5,0) and +(-3,1)  .. (0,-1.5);
  \node[draw,scale=.6] at (B) {};
  \draw (B) node[right,scale=.6]{$+\frac32$};
  \fill[white] (BW) circle (1mm); 
  \fill[white] (0,.5) circle (1mm);
  \draw[green!50!black]  (0,1.5) -- (0,0) coordinate[pos=0.4] (G);
  \node[draw,scale=.6] at (G) {};
  \draw (G) node[right,scale=.6]{$-\frac52$};
 \draw[red]  (0,-1.5) .. controls +(3,1) and +(.5,0) .. (1.5,.3);
\end{scope}
\end{scope}
\end{scope}
\end{scope}}
  \caption{Movie describing a spanning foam for the Kinoshita knotted graph}
\label{fig:kinmov}
\end{figure}
In between the successive frames of the movie, one has a canonical foamy cobordism. The spanning foam $F$ of $\Gamma$ is 
obtained by composing all these cobordisms together and finally glue this foam with a trivial half-theta foam. 
This gives a foam whose boundary is $\Gamma$. 

Since all sub-links of $\Gamma$ are trivial, we have $\sigma(\Gamma) =0$. In order to determine the other signature invariants it is enough to compute $e(F_{ab})$, $e(F_{bc})$, $e(F_{ac})$, $e(\widetilde{F}_{a}^{a})$, $e(\widetilde{F}_{b}^{b})$ and $e(\widetilde{F}_{c}^{c})$. 

As explained in Section~\ref{sec:normal-euler-numbers-1}, normal Euler numbers of surfaces with boundary can be computed via some linking numbers. We first inspect the surfaces $F_{ij}$.

The standard theta diagram can be seen as a framed graph equipped 
with three specific parallels (see the first frame in Figure~\ref{fig:kinmov}). Following these parallels during the various steps of the movie gives for each pair $\{i,j\}$ a {\it non-vanishing} section of the normal bundle of the surface (with boundary) $F_{ij}$. The boxes with integers or half-integers drawn in the movie encode twists or half-twists between an edge and its parallel.  

The normal Euler number $e(F_{ij})$ is equal to $-l_{ij}$,  where $l_{ij}$ is the linking number in $\SS^3$ of the knot $\Gamma_{ij}=\partial F_{ij}$ with its parallel 
(see Definition~\ref{dfn:QHSeulernumbers} and Remark~\ref{rmk:NENwithlk} (3)): that is the sum of all values of the boxes in the last frame plus the number of crossings counted  
algebraically. 

Note that between the frames 6 and 7, a box on a green ($c$) edge travels the other green edge (because we need to unzip the 
first edge). This does not impact the validity of the computation.

We have:
\begin{itemize}
\item $e({F_{ab}}) = -(1 +\frac32 + \frac32)=-4$,
\item $e({F_{bc}}) = e({F_{ac}})= -(1+\frac32-\frac52)=0$.
\end{itemize}

We now explain how to compute $e(\widetilde{F}_i^i)$ for $i$ in $\{a,b,c\}$. As we will see, it is possible to deduce 
$e(\widetilde{F}_a^a)$ and $e(\widetilde{F}_b^c)$ from $e(\widetilde{F}_c^c)$ by  using symmetries. For computing 
$e(\widetilde{F}_c^c)$,  we consider a section of the normal bundle of $\widetilde{F}_c^c$ (in $W_F/c$) transverse to the trivial 
section. 
Such a section can be read on the movie. Indeed, consider each step of the movie as a framed graph where every edge colored by $c$ comes 
with a parallel attached to its adjacent edges\footnote{Here the apparent choice of side has no effect.}. 
The pre-image of these parallels in the double cover along the knot consisting of edges colored by $a$ and $b$ 
is a genuine parallel of the pre-image of the edge, as shown in the following picture:
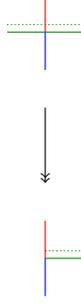
\[
  \tikz[scale=1]{\begin{scope}[xscale =0.5, yscale = 0.5]
\newcommand{\localcoordinate}{
 \coordinate (T) at (0,1);
  \coordinate (TL) at (-1,2);
   \coordinate (tL) at (-1,1);
  \coordinate (TR) at (1,2);
   \coordinate (tR) at (1,1);
  \coordinate (B) at (0,-1);
  \coordinate (R) at (1,0);
  \coordinate (O) at (0,0);
   \coordinate (L) at (-1,0);
   \coordinate (bL) at (-1,-1);
   \coordinate (BL) at (-1,-2);
    \coordinate (bR) at (1,-1);
   \coordinate (BR) at (1,-2);
}

\begin{scope}[yshift = 6cm]
\localcoordinate
 \draw[red] (T) -- (O);
 \draw[blue] (O) -- (B);
 \draw[black!50!green] (L) --(R);
 \draw[black!50!green,densely dotted] (-1,.2) --(1,.2);
\end{scope}

\begin{scope}[yshift = 3cm]
 \localcoordinate
 \draw[->>] (T) -- (B);
\end{scope}

\begin{scope}[yshift = 0cm]
 \localcoordinate
 \draw[red] (T) -- (O);
 \draw[blue] (O) -- (B);
 \draw[black!50!green] (O) --(R);
 \draw[black!50!green,densely dotted] (0,.2) --(1,.2);
 \end{scope}

\end{scope}
}
\]
Following these parallels during the movie gives an appropriate section of $\widetilde{F}_c^c$ in $W_{F}/c$.
The surface $\widetilde{F}_c^c$ has two connected components: a knotted sphere $S$ coming from the small horizontal green edge on the frames $5$, $6$ and $7$ and another component $\Sigma$ whose boundary is the knot $\widetilde{\Gamma}_c^c$.

Let us first deal with the knotted sphere $S$. We isolate the interesting part of the movie:
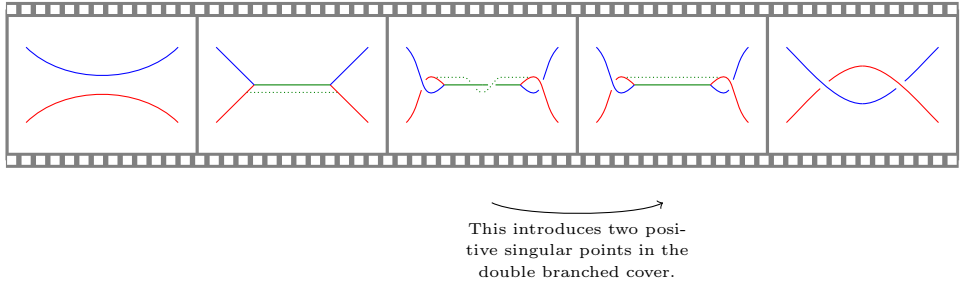
\[
\NB{
\tikz{\begin{scope}[xscale =0.5, yscale = 0.5]
  \draw[draw=  gray, line width = 2mm] (-0.03, 2) --  (25.03, 2); 
  \draw[draw=  gray, line width = 2mm] (-0.03,-2) --  (25.03,-2); 
  \draw[draw=  gray, very thick] (0, -2) -- +(0,4);
  \draw[draw=  gray, very thick] (5, -2) -- +(0,4);
  \draw[draw=  gray, very thick] (10,-2) -- +(0,4);
  \draw[draw=  gray, very thick] (15,-2) -- +(0,4);
  \draw[draw=  gray, very thick] (20,-2) -- +(0,4);
  \draw[draw=  gray, very thick] (25,-2) -- +(0,4);
  \draw[draw= white, dotted, line width =1.2mm] (0,2) --  (25,2); 
  \draw[draw= white, dotted, line width = 1.2mm] (0,-2) --  (25,-2); 
\newcommand{\localcoordinate}{
 \coordinate (T) at (0,1);
  \coordinate (TL) at (-1,2);
  \coordinate (TR) at (1,2);
  \coordinate (B) at (0,-1);
  \coordinate (O) at (0,0);
  \coordinate (BL) at (-1,-2);
  \coordinate (BR) at (1,-2);
}
\begin{scope}[xshift = 2.5cm, rotate =90]
 \localcoordinate
  \draw[black!50!green] [red] (TL)  .. controls +(1, -1) and +(1, 1) .. (BL);
  \draw[blue] (TR) .. controls +(-1, -1) and +(-1, 1) .. (BR);
\end{scope}
\begin{scope}[xshift = 7.5cm, , rotate =90]
\localcoordinate
   \draw[black!50!green] (T) -- (B);
   \draw[black!50!green, densely dotted] (T)+(-0.2,0.2) -- ($(B) +(-0.2, -0.2)$);
   \draw[red] (T) -- (TL);
   \draw[red] (B) -- (BL);
   \draw[blue] (T) -- (TR);
   \draw[blue] (B) -- (BR);
\end{scope}
\begin{scope}[xshift = 12.5cm, rotate =90]
\localcoordinate
   \draw[black!50!green, densely dotted] ($(T)+( 0.2, 0.2)$) --  ($(O) +(0.2, 0.5)$).. controls +(0,-0.2) and +(0, 0.2).. ($(O) + (-0.2,0)$);
   \draw[black!50!green] (B) -- (O);
   \fill[white] ($(O) + (0,-0.25)$) circle (1mm);
   \fill[white] ($(O) + (0,0.25)$) circle (1mm);
   \draw[black!50!green] (T) -- (O);
   \draw[black!50!green, densely dotted] ($(B)+( 0.2,-0.2)$) --  ($(O) +(0.2,-0.5)$).. controls +(0, 0.2) and +(0,-0.2).. ($(O) + (-0.2,0)$);
   \draw[red] (T) ..  controls +(0.7, 0.7) and + (0.5,-0.5)  ..  (TL);
   \fill[white] ($(T) + (0,0.55)$) circle (1.5mm);
   \draw[blue] (T)..   controls +(-0.7, 0.7) and + (-0.5,-0.5)   ..  (TR);
   \draw[blue] (B)..   controls +(-0.7, -0.7) and + (-0.5,0.5)   ..  (BR);
   \fill[white] ($(B) + (0,-0.55)$) circle (1.5mm);
   \draw[red] (B) ..  controls +(0.7, -0.7) and + (0.5,0.5)  ..  (BL);
   \node (A1) at (-3, 0) {};
\end{scope}
\begin{scope}[xshift = 17.5cm, , rotate =90]
  \localcoordinate
   \draw[black!50!green] (B) -- (T);
   \draw[black!50!green, densely dotted] (T)+(0.2,0.2) -- ($(B) +(0.2, -0.2)$);
   \draw[red] (T) ..  controls +(0.7, 0.7) and + (0.5,-0.5)  ..  (TL);
   \fill[white] ($(T) + (0,0.55)$) circle (1.5mm);
   \draw[blue] (T)..   controls +(-0.7, 0.7) and + (-0.5,-0.5)   ..  (TR);
   \draw[blue] (B)..   controls +(-0.7, -0.7) and + (-0.5,0.5)   ..  (BR);
   \fill[white] ($(B) + (0,-0.55)$) circle (1.5mm);
   \draw[red] (B) ..  controls +(0.7, -0.7) and + (0.5,0.5)  ..  (BL);
   \node (A2) at (-3, 0) {};
\end{scope}
\draw[->] (A1) .. controls +(1,-0.5) and +(-1,-0.5) .. (A2) node[midway, below, text width = 3cm, text centered, font=\tiny] {{This introduces two positive singular points in the double branched cover.}};
\begin{scope}[xshift = 22.5cm, , rotate =90]
   \localcoordinate
   \draw[red] ($(O)+(0.5, 0)$) ..  controls +(0, 0.7) and + (0.5,-0.5)  ..  (TL);
   \fill[white] (T) circle (1.5mm);
   \draw[blue] ($(O)+(-0.5, 0)$) ..   controls +(0, 0.7) and + (-0.5,-0.5)   ..  (TR);
   \draw[blue] ($(O)+(-0.5, 0)$)..   controls +(0, -0.7) and + (-0.5,0.5)   ..  (BR);
   \fill[white] (B) circle (1.5mm);
   \draw[red] ($(O)+(0.5, 0)$) ..  controls +(0, -0.7) and + (0.5,0.5)  ..  (BL);
\end{scope}
\end{scope}}
}
\label{page:clasp}
\]
Hence we have $e(S) = +2$.

In order to compute $e(\Sigma)$, we consider the section of $\Sigma$ given by the pre-image of the parallels. This section does not intersect $\Sigma$, hence $e(\Sigma)$ is equal to $-\ell_c$ where $\ell_c$ is equal to the linking number of the knot $\widetilde{\Gamma}_c^c$ with the pre-image  in $\partial W_F/c$ of its preferred parallel of the edge of $\Gamma$ colored by $c$. 

For computing $\ell_c$, we use the following theorem:
\begin{thm}[{\cite[Theorem 1.1]{PY}}, restricted to our case]\label{thm:PY}
  Let $K_1$ and $K_2$ be two knots in $\mathbb{S}^3$, $J=(J_1, \dots J_l)$ a link in $\mathbb{S}^3$ disjoint from $K_1$ and $K_2$ 
  and $r_1, \dots r_l$ some rational numbers. 
  Let $M$ be the manifold obtained by Dehn surgery along $J$ with coefficients $r_1, \dots r_l$. Suppose that $M$ is a rational homology sphere, then
\begin{align*}
&\lk_M(K_1, K_2)- \lk_{\mathbb{S}^3}(K_1, K_2) \\ & \qquad = - \left(\lk_{\mathbb{S}^3}(K_1,J_1), \dots, \lk_{\mathbb{S}^3}(K_1,J_l)\right)G^{-1} \left(\lk_{\mathbb{S}^3}(K_2,J_1), \dots, \lk_{\mathbb{S}^3}(K_2,J_l)\right)^t,
\end{align*}
where $G= (g_{ij})_{\substack{1\leq i \leq l \\ 1 \leq j \leq l}}$ is the $l\times l$ matrix defined by:
\[
g_{ij} =
\begin{cases}
  \lk_{\mathbb{S}^3}(J_i, J_j) & \textrm{if $i\neq j$,} \\
  r_i & \textrm{if $i=j$.}
\end{cases}
\]
\end{thm}

Indeed, Montesinos'trick \cite[Section 2]{MR0380802} tells us that if one changes a knot by the local moves: 
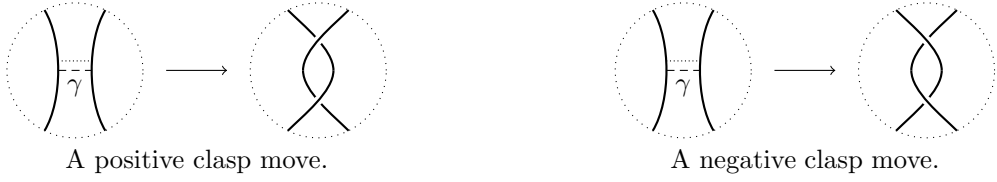
\[
   \tikz[scale=1]{\begin{scope}[xscale =0.4, yscale = 0.4]
\newcommand{\localcoordinate}{
 \coordinate (T) at (0,1);
  \coordinate (TL) at (-1,2);
  \coordinate (TR) at (1,2);
   \coordinate (tR) at (1,1.2);
  \coordinate (R) at (1,0);
  \coordinate (r) at (0.5,0);
   \coordinate (O) at (0,0);
     \coordinate (l) at (-0.5,0);
   \coordinate (L) at (-1,0);
    \coordinate (BL) at (-1,-2);
  \coordinate (BR) at (1,-2);
    \coordinate (B) at (0,-1);
}

\begin{scope}
 \localcoordinate
  \draw [thick](TL)  .. controls +(0.6, -1) and +(.6, 1) .. (BL) coordinate[midway] (ML) coordinate[pos=0.43] (MTL);
  \draw [thick] (TR) .. controls +(-0.6, -1) and +(-0.6, 1) .. (BR) coordinate[midway] (MR) coordinate[pos=0.43] (MTR);
  \draw[densely dashed] (ML) -- (MR) node[midway, below] {$\gamma$};
  \draw[densely dotted] (MTL) -- (MTR);
   \draw[dotted] (O) circle ({sqrt(5)});
\end{scope}

\begin{scope}[xshift = 4cm]
 \localcoordinate
 \draw[->] (L) -- (R);
\end{scope}

\begin{scope}[xshift = 8cm]
   \localcoordinate
   \draw[thick] ($(O)+(0.5, 0)$) ..  controls +(0, 0.7) and + (0.5,-0.5)  ..  (TL);
   \draw[thick] ($(O)+(-0.5, 0)$)..   controls +(0, -0.7) and + (-0.5,0.5)   ..  (BR);
   \fill[white] (T) circle (1.5mm);
   \fill[white] (B) circle (1.5mm);
   \draw[thick] ($(O)+(0.5, 0)$) ..  controls +(0, -0.7) and + (0.5,0.5)  ..  (BL);
   \draw[thick] ($(O)+(-0.5, 0)$) ..   controls +(0, 0.7) and + (-0.5,-0.5)   ..  (TR);
      \draw[dotted] (O) circle ({sqrt(5)});

\end{scope}
\node at (4, -3) {A positive clasp move.};

\begin{scope}[xshift =20cm]
\begin{scope}
 \localcoordinate
  \draw [thick](TL)  .. controls +(0.6, -1) and +(.6, 1) .. (BL) coordinate[midway] (ML) coordinate[pos=0.43] (MTL);;
  \draw [thick] (TR) .. controls +(-0.6, -1) and +(-0.6, 1) .. (BR) coordinate[midway] (MR) coordinate[pos=0.43] (MTR);;
  \draw[densely dashed] (ML) -- (MR) node[midway, below] {$\gamma$};
  \draw[densely dotted] (MTL) -- (MTR);
   \draw[dotted] (O) circle ({sqrt(5)});
 \end{scope}

\begin{scope}[xshift = 4cm]
 \localcoordinate
 \draw[->] (L) -- (R);
\end{scope}

\begin{scope}[xshift = 8cm]
   \localcoordinate
   \draw[thick] ($(O)+(-0.5, 0)$) ..   controls +(0, 0.7) and + (-0.5,-0.5)   ..  (TR);
   \draw[thick] ($(O)+(0.5, 0)$) ..  controls +(0, -0.7) and + (0.5,0.5)  ..  (BL);
   \fill[white] (T) circle (1.5mm);
   \fill[white] (B) circle (1.5mm);
   \draw[thick] ($(O)+(0.5, 0)$) ..  controls +(0, 0.7) and + (0.5,-0.5)  ..  (TL);
   \draw[thick] ($(O)+(-0.5, 0)$)..   controls +(0, -0.7) and + (-0.5,0.5)   ..  (BR);
      \draw[dotted] (O) circle ({sqrt(5)});

\end{scope}
\node at (4, -3) {A negative clasp move.};
\end{scope}
\end{scope}} 
\]
the double branched cover of the new knot is obtained by a surgery along $\widetilde{\gamma}$ with coefficient 
$-\frac12$ for a positive clasp move, $+\frac12$ for a negative clasp move, where $\widetilde{\gamma}$ is the 
pre-image of $\gamma$ in the double branched cover. These coefficients are given in the 
canonical\footnote{The longitude is required to be the pre-image of the dotted arc.} basis longitude/meridian of $\widetilde{\gamma}$.  

Since the graph (denoted by $\Gamma'$) in the 4th frame is still a trivial theta graph, all double branched 
cover are diffeomorphic to $\SS^3$. The following picture describes the branched double cover of $\SS^3$ along 
$\Gamma'_{ab}$:
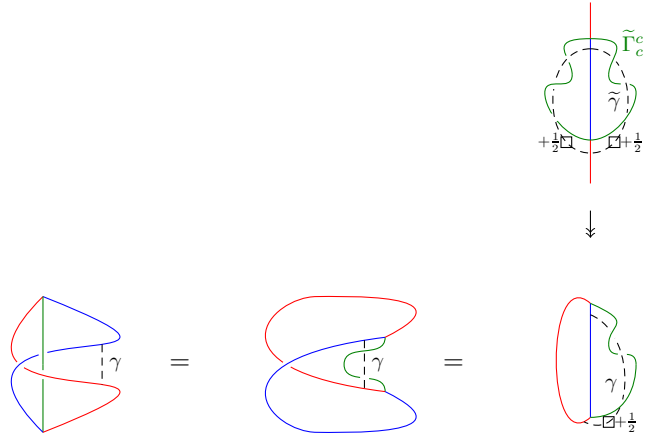
\[
\NB{
\tikz[scale=0.6]{\begin{scope}
    \begin{scope}  [xshift=0cm] 
      \draw[green!50!black]  (0,0) -- (0,-1.5) coordinate[pos=0.15] (GW);
      \fill[white] (GW) circle (1mm);
      \draw[red]  (0,-1.5) .. controls +(5,2) and +(-3,-3) .. (0,1.5) coordinate[pos=0.9] (R) coordinate[pos=0.73] (RW) coordinate[pos=0.4] (RC);
      \fill[white] (RW) circle (1mm);
      \draw[blue]  (0,-1.5) .. controls +(-3,3) and +(5,-2) .. (0,1.5) coordinate[pos=0.9] (B) coordinate[pos=0.375] (BW) coordinate[pos=0.6] (BC);
      \fill[white] (BW) circle (1mm);
      \draw[green!50!black]  (0,1.5) -- (0,0) coordinate[pos=0.5] (G);
      \draw[densely dashed] (BC)  --  (RC) coordinate[pos=0.5] (M);
      \draw (M) node[right, scale=.8] {$\gamma$};
    \end{scope}
    \begin{scope}[xshift=6cm]
      \draw[green!50!black]  (1.5,0.6) .. controls +(0,-.5) and +(0,.5) .. (.6,0) .. controls +(0,-.5) and +(0,.5) .. (1.5,-.6) coordinate[pos=0.5] (GW);
      \fill[white] (GW) circle (1mm);
      \draw[red]  (1.5,-.6) .. controls +(-4,.5) and +(-1,0) .. (0,1.5) coordinate[pos=0.04] (RC) coordinate[pos=0.27] (RW)
      ..controls +(.5,0) and +(2,1).. (1.5,.6) coordinate[pos=0.9] (R)  ;
      \fill[white] (RW) circle (1mm);
      \draw[blue]  (1.5,.6) .. controls +(-4,-.5) and +(-1,0) .. (0,-1.5) coordinate[pos=0.04] (BC) .. controls +(.5,0) and +(2,-1).. (1.5,-.6) coordinate[pos=0.9] (B)  ;
      \draw[densely dashed] (BC)  --  (RC) coordinate[pos=0.5] (M);
      \draw (M) node[right, scale=.8] {$\gamma$};
    \end{scope}
    \begin{scope}[xshift=12cm, yscale =0.9]
      \draw[green!50!black]   (.4,.5) .. controls +(0,-.5) and +(0,.5) .. (1,0) coordinate[pos=0.5] (GW);
      \fill[white] (GW) circle (1mm);
      \draw[red]  (0,1.5) .. controls +(-1,1) and +(-1,-1) .. (0,-1.3) coordinate[pos=0.95] (RC) ;      
      \draw[blue]  (0,1.5) --  (0,-1.3) coordinate[pos=0.1] (BC) ;
      \draw[densely dashed] (BC)  .. controls +(1.1,-.5) and +(1.1,-0.6).. (RC) coordinate[pos=0.8] (D) coordinate[pos=0.5] (M) coordinate[pos=0.19] (DW) coordinate[pos=0.62] 
      (DWW);
      \draw (M) node[left, scale=.8] {$\gamma$};
      \node[draw,scale=.6] at (D) {};
      \draw (D) node[right,scale=.6]{$+\frac12$};
      \fill[white] (DW) circle (1mm);
      \fill[white] (DWW) circle (1mm);
      \draw[green!50!black]  (0,1.5) .. controls +(1,-.5) and +(0,.5) .. (.4,.5); 
      \draw[green!50!black] (1,0) .. controls +(0,-.5) and +(0.8,0) .. (0,-1.3);
    \end{scope}
\node at (3,0) {$=$};
\node at (9,0) {$=$};
\draw[->>] (12,3.4) -- (12, 2.8);
  \begin{scope}[yshift=6cm, yscale=0.8]
      \begin{scope}[xshift=12cm]
        \draw[red]  (0,1.5) -- (0, 2.5) ;      
        \draw[red]  (0,-1.3) --(0,-2.5) coordinate[pos=0.3] (RC) ;      
        \draw[blue]  (0,1.5) --  (0,-1.3) coordinate[pos=0.1] (BC) ;
        \draw[green!50!black] (.4,.5) .. controls +(0,-.5) and +(0,.5) .. (1,0) coordinate[pos=0.57] (GW);
        \fill[white] (GW) circle (1mm);
        \draw[densely dashed] (BC)  .. controls +(1.1,0) and +(1.1,0).. (RC) coordinate[pos=0.8] (D) coordinate[pos=0.5] (M) coordinate[pos=0.19] (DW) coordinate[pos=0.6] (DWW);
        \draw (M) node[left, scale=.8] {$\widetilde{\gamma}$};
        \node[draw,scale=.6] at (D) {};
        \draw (D) node[right,scale=.6]{$+\frac12$};
        \fill[white] (DW) circle (1mm);
        \fill[white] (DWW) circle (1mm);
        \draw[green!50!black] (0,1.5) .. controls +(1,0) and +(0,.5) .. (.4,.5) coordinate[pos=0.3] (N) ; 
        \draw[green!50!black] (N) node[right, scale=.8] {$\widetilde{\Gamma}^{c}_{c}$};      
        \draw[green!50!black]  (0,1.5) .. controls +(-1,0) and +(-0,.5) .. (-.4,.5) coordinate[pos=0.7] (GW);
        \fill[white] (GW) circle (1mm); 
        \draw[green!50!black] (1,0) .. controls +(0,-.5) and +(.5,0) .. (0,-1.3);       
        \draw[green!50!black] (-1,0) .. controls +(0,-.5) and +(-.5,0) .. (0,-1.3)coordinate[pos=0.4] (GWW);
        \fill[white] (GWW) circle (1.2mm); 
        \draw[densely dashed] (BC)  .. controls +(-1.1,0) and +(-1.1,0) .. (RC) coordinate[pos=0.8] (D) coordinate[pos=0.5] (M) coordinate[pos=0.4] (DW);
        \node[draw,scale=.6] at (D) {};
        \draw (D) node[left,scale=.6]{$+\frac12$};
        \fill[white] (DW) circle (1mm);
        \draw[green!50!black] (-.4,.5) .. controls +(0,-.5) and +(0,.5) .. (-1,0) ; 
      \end{scope}
    \end{scope}
\end{scope}}
}
\]

The knot $\widetilde{\gamma}$ is a framed trivial knot with framing +1. In Montesinos' trick, the surgery coefficient is $-\frac{1}{2}$ for the basis given by the framing. For the standard basis it is therefore $\frac{-1 + 2}{2} = \frac{1}{2}$. Moreover, we have $\lk_{\mathbb{S}^3}(\widetilde{\Gamma'}^{c}_{c},\widetilde{\gamma})= \pm 3$ (the sign depends on which orientation we take for $\widetilde{\gamma}$). Thanks to Theorem~\ref{thm:PY}, we have: $\ell_c = -(\pm3)^2(\frac12)^{-1}=-18$ and finally $e(\widetilde{F}^{c}_c)= 2 + 18 = 20$.

The remaining normal Euler numbers are $e(\widetilde{F}^{a}_a)$ and $e(\widetilde{F}^{b}_b)$.

First of all, the three bi-colored knots of $\Gamma$ are trivial so that 
$\sigma(\Gamma_{ab})=\sigma(\Gamma_{bc})=\sigma(\Gamma_{ca})=0$.
Moreover because of the  symmetry of $\Gamma$,  the knots 
$\widetilde{\Gamma}^{a}_{a}$, $\widetilde{\Gamma}^{b}_{b}$ and $\widetilde{\Gamma}^{c}_{c}$ are the same. Denote by $s$ the signature of this knot. By symmetry of the movie relatively to the colors $a$ and $b$
we have $e(\widetilde{F}^{a}_a)=e(\widetilde{F}^{b}_b)$. Denote by $e$ this value.
We have: 
\begin{align*}
  \begin{cases}
    \delta_{ab}(\Gamma)=\delta_{bc}(\Gamma)=\delta_{ca}(\Gamma)=\frac12 s +\frac12 s -0=s \\
    \delta_{ab}(\Gamma)=\frac14e(\widetilde{F}^{a}_a)+\frac14e(\widetilde{F}^{b}_b)
 -\frac12e(F_{ab})=\frac12 e +2, \\
 \delta_{bc}(\Gamma)=\frac14e(\widetilde{F}^{b}_b)+\frac14e(\widetilde{F}^{c}_c)
 -\frac12e(F_{bc})=\frac14 e +5.
  \end{cases}
\end{align*}
It implies that $e=12$ and $s=8$.
In conclusion we have: 
\begin{align*}
  \begin{cases}
    \delta_{ab}(\Gamma)=\delta_{bc}(\Gamma)= \delta_{ac}(\Gamma) = 8, \\
    \sigma (\Gamma) = 0, \\ 
     \widetilde{\sigma}(\Gamma)=24.
  \end{cases}
\end{align*}

In particular the Kinoshita graph is not trivial.

\begin{rmk}
As an intermediate result, our computations give the signature of the knot $\widetilde{\Gamma}^{a}_{a}$
in $\SS^3$ (which is $8$). Note that we did not need to determine this knot which happens to be the mirror image of
$10_{124}$ (see for example \cite{MR3530305}).
\end{rmk}

\begin{figure}[ht]
  \centering
  \tikz{\input{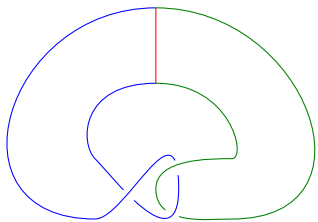}}
  \caption{The trefoil theta graph.}
  \label{fig:trefoil}
\end{figure}

The same kind of computations can be done for other knotted graphs. For the trefoil theta, we find:

\begin{align*}
  \begin{cases}
    \delta_{ab}(\Gamma)=\delta_{ca}(\Gamma)= \frac43 \\
    \delta_{bc}(\Gamma)=4 \\
    \sigma (\Gamma) = -2, \\ 
     \widetilde{\sigma}(\Gamma)=\frac{14}{3}.
  \end{cases}
\end{align*}

As for the Kinoshita graph, the symmetries give an intermediate result: the signature of $\sigma(\widetilde{\Gamma}^{a}_{a})$ which is a knot in the lens space $L(3,1)$ is equal to $\frac{2}{3}$.

\appendix

\section{Spanning foams}
\label{sec:spanning-foams}
The aim of this section is to prove the following statement:
\begin{prop}\label{prop:spanningfoamexists}
  Let $\Gamma$ be a knotted Klein graph in $\SS^3$. There exists a spanning foam for $\Gamma$ in $\BB^4$.
\end{prop}

We need to introduce some elementary foamy cobordisms. A \emph{clasp} is a cobordism which performs a bicolor crossing change. It is described by the movie at the bottom of page~\pageref{page:clasp}. An \emph{unzip} is a simply the second part of a clasp. It removes two vertices of the graph. Locally an \emph{unzip} looks like the picture drawn below.
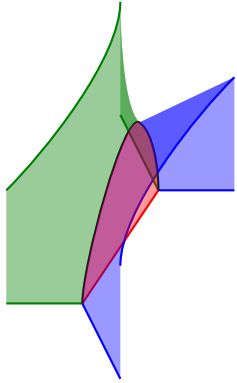
\[
\NB{\tikz[scale=0.5]{\begin{scope}[xshift=10cm, yshift= 0cm, decoration={markings, mark=at
     position 0.5 with {\arrow{>}}},postaction={decorate}]
\coordinate (A2) at (1,5);
\coordinate (B2) at (4,3);
\coordinate (C2) at (-2,0);
\coordinate (D2) at (1,-2);
\coordinate (E2) at (0,-3);
\coordinate (F2) at (2,0);
\coordinate (A1) at (1,2);
\coordinate (B1) at (4,0);
\coordinate (C1) at (-2,-3);
\coordinate (D1) at (1,-5);
\coordinate (G2) at (1.5,1.8);
\draw[thick, black!50!green] (A1) -- (F2);
\draw[thick, blue] (F2) -- (B1);
\draw[thick, black!50!green] (C1) -- (E2); 
\draw[thick, blue] (E2) -- (D1);
\draw[thick, red] (E2) -- (F2);
\draw[thick, black!50!green] (C2) .. controls +(1,1) and +(0,-1.5) .. (A2);
\draw[thick, blue] (D2) .. controls +(0,1.5) and +(-1,-1) .. (B2);
\draw[thick, black] (E2) .. controls +(0,1) and  +(-0.4,0.3).. (G2).. controls +(0.4,-0.3) and +(0,0.5) .. (F2);
\fill[opacity=0.4, color = black!50!green] (A2) .. controls +(0,0) and +( -0.5, +0.5) ..  (G2).. controls +(0.4,-0.3) and +(0,0.5) .. (F2) -- (A1)-- (A2);
\fill[opacity=0.4, color = black!50!green] (C2) -- (C1) --(E2)  .. controls +(0,1) and  +(-0.4,0.3).. (G2) .. controls +(-0.5,0.5)  and + (0,0) .. (A2)  .. controls +(0,-1.5) and +(1,1) ..  (C2);;
\fill[opacity=0.4, color = blue] (B2) -- (G2)  .. controls +(0.4,-0.3) and  +(0,0).. (F2) -- (B1) --(B2);
\fill[opacity=0.4, color = blue] (B2)  .. controls +(-1,-1) and  +(0,1.5).. (D2) -- (D1) -- (E2)  .. controls +(0,1) and  +(-0.4,0.3).. (G2)--(B2);
\fill[opacity=0.4, color= red] (E2) .. controls +(0,1) and  +(-0.4,0.3).. (G2).. controls +(0.4,-0.3) and +(0,+0.5) .. (F2) --(E2);
\end{scope}}}
\]
\begin{lem}
  \label{lem:cob-gamma-2-link}
  Let $\Gamma$ be a knotted Klein graph in $\SS^3$. There exits a cobordism $F$ in $\SS^3\times I$ from $\Gamma$ to a Klein graph with no vertices (i.~e.~a link).
\end{lem}

\begin{proof}
Let $D$ be a diagram for $\Gamma$, we consider all edges\footnote{Here, we mean \emph{real} edges, i.~e.~not circles.} of $D$ colored by $\aa$, with a sequence of Reidemeister move, we can shrink these edges until they are not involved in any crossing. Then we unzip these edges. Before performing the unzip, we may have to twist one end of theses edges with the Rv1 Reidemeister move.
\end{proof}

\begin{lem}
  \label{lem:link-2-unlink}
  Let $\Gamma$ be a knotted Klein graph in $\SS^3$ with no vertices. This can be seen as a link $L$ colored by $\kg^*$, this means that $L= L_\aa \cup L_\bb \cup L_\cc$. There exists a (foamy) cobordism from $L$ to a link $L'= L'_\aa\cup L'_\bb \cup L'_\cc$, where the component $L'_\aa$,  $L'_\bb$ and $L'_\cc$ are in three disjoint balls. 
\end{lem}

\begin{proof}
  In order to unlink the component $L_\aa$, $L_\bb$ and $L_\cc$ of $L$, it is enough to perform crossing changes on bi-colored crossings. These crossing changes can be achieved by clasps. 
\end{proof}

\begin{proof}[Proof of Proposition~\ref{prop:spanningfoamexists}]
  Thanks to Lemmas~\ref{lem:cob-gamma-2-link} and \ref{lem:link-2-unlink} we can construct a cobordism 
  $F$ in $\SS^3\times I$ from $\Gamma$ to a link $L'=L'_\aa\cup L'_\bb \cup L'_\cc$, where the component $L'_\aa$,  $L'_\bb$ 
  and $L'_\cc$ are in disjoint balls. We now pick  Seifert surfaces for $L'_\aa$,  $L'_\bb$ and $L'_\cc$, push them in $\BB^4$ and concatenate them with $F$. This gives a spanning foam for $\Gamma$. 
\end{proof}

\bibliographystyle{alphaurl}
\bibliography{biblio}

\newcommand{\etalchar}[1]{$^{#1}$}
\begin{thebibliography}{JKL{\etalchar{+}}16}

\bibitem[AS68]{MR0236952}
Michael~F. Atiyah and Isadore~M. Singer.
\newblock The index of elliptic operators. {III}.
\newblock {\em Ann. of Math. (2)}, 87:546--604, 1968.
\newblock URL: \url{https://doi.org/10.2307/1970717}.

\bibitem[CMB16]{MR3530305}
Jack~S. Calcut and Jules~R. Metcalf-Burton.
\newblock Double branched covers of theta-curves.
\newblock {\em J. Knot Theory Ramifications}, 25(8):1650046, 9, 2016.
\newblock URL: \url{https://doi.org/10.1142/S0218216516500462}.

\bibitem[Fla00]{MR1781912}
Erica Flapan.
\newblock {\em When topology meets chemistry}.
\newblock Outlooks. Cambridge University Press, Cambridge; Mathematical
  Association of America, Washington, DC, 2000.
\newblock A topological look at molecular chirality.
\newblock URL: \url{https://doi.org/10.1017/CBO9780511626272}.

\bibitem[Gil92]{MR1145914}
Patrick Gilmer.
\newblock Real algebraic curves and link cobordism.
\newblock {\em Pacific J. Math.}, 153(1):31--69, 1992.
\newblock URL: \url{http://projecteuclid.org/euclid.pjm/1102635971}.

\bibitem[Gil93]{MR1238876}
Patrick~M. Gilmer.
\newblock Link cobordism in rational homology {$3$}-spheres.
\newblock {\em J. Knot Theory Ramifications}, 2(3):285--320, 1993.
\newblock URL: \url{https://doi.org/10.1142/S0218216593000179}.

\bibitem[GL78]{GL78}
Cameron.~McA. Gordon and Richard.~A. Litherland.
\newblock On the signature of a link.
\newblock {\em Invent. Math.}, 47(1):53--69, 1978.
\newblock URL: \url{https://doi.org/10.1007/BF01609479}.

\bibitem[GL79]{GL79}
Cameron.~McA. Gordon and Richard.~A. Litherland.
\newblock On a theorem of {M}urasugi.
\newblock {\em Pacific J. Math.}, 82(1):69--74, 1979.
\newblock URL: \url{https://doi.org/10.2140/pjm.1979.82.69}.

\bibitem[Gor86]{GSign}
Cameron.~McA. Gordon.
\newblock On the {$G$}-signature theorem in dimension four.
\newblock In {\em \`{A} la recherche de la topologie perdue}, volume~62 of {\em
  Progr. Math.}, pages 159--180. Birkh\"auser Boston, Boston, MA, 1986.
\newblock URL: \url{http://data.rero.ch/01-0693409/html}.

\bibitem[JKL{\etalchar{+}}16]{MR3541985}
Byoungwook Jang, Anna Kronaeur, Pratap Luitel, Daniel Medici, Scott~A. Taylor,
  and Alexander Zupan.
\newblock New examples of {B}runnian theta graphs.
\newblock {\em Involve}, 9(5):857--875, 2016.
\newblock URL: \url{https://doi.org/10.2140/involve.2016.9.857}.

\bibitem[Kin58]{MR0102819}
Shin'ichi Kinoshita.
\newblock Alexander polynomials as isotopy invariants. {I}.
\newblock {\em Osaka Math. J.}, 10:263--271, 1958.
\newblock URL: \url{https://doi.org/10.18910/6353}.

\bibitem[Kin72]{MR0312485}
Shin'ichi Kinoshita.
\newblock On elementary ideals of polyhedra in the {$3$}-sphere.
\newblock {\em Pacific J. Math.}, 42:89--98, 1972.
\newblock URL: \url{http://projecteuclid.org/euclid.pjm/1102968011}.

\bibitem[Kir89]{MR1001966}
Robion~C. Kirby.
\newblock {\em The topology of {$4$}-manifolds}, volume 1374 of {\em Lecture
  Notes in Mathematics}.
\newblock Springer-Verlag, Berlin, 1989.
\newblock URL: \url{https://doi.org/10.1007/BFb0089031}.

\bibitem[KT76]{KT76}
Louis~H. Kauffman and Laurence~R. Taylor.
\newblock Signature of links.
\newblock {\em Trans. Amer. Math. Soc.}, 216:351--365, 1976.
\newblock URL: \url{https://doi.org/10.2307/1997704}.

\bibitem[Les15]{MR3586621}
Christine Lescop.
\newblock An introduction to finite type invariants of knots and 3-manifolds
  defined by counting graph configurations.
\newblock {\em Vestn. Chelyab. Gos. Univ. Mat. Mekh. Inform.}, 3(17):67--117,
  2015.
\newblock URL: \url{http://mi.mathnet.ru/vchgu8}.

\bibitem[Lew13]{lewar2013}
Lukas Lewark.
\newblock {\em {H}omologies de {K}hovanov--{R}ozansky, toiles nouées
  pondérées et genre lisse}.
\newblock PhD thesis, Université Paris 7 -- Denis Diderot, 2013.
\newblock {T}hèse de doctorat dirigée par {C}hristian {B}lanchet.
\newblock URL: \url{http://www.theses.fr/2013PA077117}.

\bibitem[Mon75]{MR0380802}
Jos\'e~M. Montesinos.
\newblock {\em Knots, groups, and 3-manifolds. {P}apers dedicated to the memory
  of {R}. {H}. {F}ox}, chapter Surgery on links and double branched covers of
  {$S^{3}$}, pages 227--259. Ann. of Math. Studies, No. 84.
\newblock Princeton Univ. Press, Princeton, N.J., 1975.

\bibitem[Mur70]{Murasugi}
Kunio Murasugi.
\newblock On the signature of links.
\newblock {\em Topology}, 9:283--298, 1970.
\newblock URL: \url{https://doi.org/10.1016/0040-9383(70)90018-2}.

\bibitem[PY04]{PY}
J\'ozef~H. Przytycki and Akira Yasuhara.
\newblock Linking numbers in rational homology 3-spheres, cyclic branched
  covers and infinite cyclic covers.
\newblock {\em Trans. Amer. Math. Soc.}, 356(9):3669--3685, 2004.
\newblock URL: \url{http://dx.doi.org/10.1090/S0002-9947-04-03423-3}.

\bibitem[Rol90]{MR1277811}
Dale Rolfsen.
\newblock {\em Knots and links}, volume~7 of {\em Mathematics Lecture Series}.
\newblock Publish or Perish, Inc., Houston, TX, 1990.
\newblock Corrected reprint of the 1976 original.

\bibitem[Sim87]{MR891813}
Jonathan Simon.
\newblock Molecular graphs as topological objects in space.
\newblock {\em J. Comput. Chem.}, 8(5):718--726, 1987.
\newblock URL: \url{https://doi.org/10.1002/jcc.540080516}.

\end{thebibliography}

\end{document}